\documentclass[11pt]{amsart}
\usepackage{amsfonts}
\usepackage{amssymb}
\usepackage{graphicx,ulem}
\usepackage{pstricks}
\usepackage{amsmath}
\usepackage{amsmath}
\usepackage{amsxtra}
\usepackage{pdfpages}
\usepackage{hyperref}

\theoremstyle{plain}
\newtheorem{theorem}{Theorem}[section]
\newtheorem{lemma}[theorem]{Lemma}
\newtheorem{proposition}[theorem]{Proposition}
\newtheorem{corollary}[theorem]{Corollary}

\newtheorem{definition}[theorem]{Definition}

\newtheorem{problem}[theorem]{Problem}

\theoremstyle{definition}

\newtheorem{example}[theorem]{Example}
\newtheorem{remark}[theorem]{Remark}

\numberwithin{equation}{subsection} 
\newcommand{\tr}{\operatorname{tr}\,}
\newcommand{\shift}{\operatorname{shift}\,}
\newcommand{\re}{\operatorname{Re}\,}

\makeatletter
\def\@tocline#1#2#3#4#5#6#7{\relax
  \ifnum #1>\c@tocdepth 
  \else
    \par \addpenalty\@secpenalty\addvspace{#2}%
    \begingroup \hyphenpenalty\@M
    \@ifempty{#4}{%
      \@tempdima\csname r@tocindent\number#1\endcsname\relax
    }{%
      \@tempdima#4\relax
    }%
    \parindent\z@ \leftskip#3\relax \advance\leftskip\@tempdima\relax
    \rightskip\@pnumwidth plus4em \parfillskip-\@pnumwidth
    #5\leavevmode\hskip-\@tempdima
      \ifcase #1
       \or\or \hskip 1em \or \hskip 2em \else \hskip 3em \fi%
      #6\nobreak\relax
    \hfill\hbox to\@pnumwidth{\@tocpagenum{#7}}\par
    \nobreak
    \endgroup
  \fi}
\makeatother

\definecolor{darkgreen}{rgb}{.1,.5,0}

\newcommand\cre[1]{{\color{red}#1}}

\newcommand\cbl[1]{{\color{blue}#1}}

\begin{document}
\title{Semi-hyponormality of commuting pairs \\ of Hilbert space
operators}
\author{Ra\'{u}l E. Curto}
\address{Department of Mathematics, The University of Iowa, Iowa City, Iowa
52242}
\email{raul-curto@uiowa.edu}
\author{Jasang Yoon}
\address{School of Mathematical and Statistical Sciences, The University of
Texas Rio Grande Valley, Edinburg, Texas 78539, USA}
\email{jasang.yoon@utrgv.edu}
\thanks{The first named author was partially funded by National Science
Foundation (U.S.) grant DMS-2247167.}

\begin{abstract}
We first find an explicit formula for the square root of positive $2 \times 2$ operator matrices with commuting entries, and then use it to define and study semi-hyponormality for commuting pairs of Hilbert space operators. \ For the well-known $3$--parameter family $W_{(\alpha,\beta)}(a,x,y)$ of $2$--variable weighted shifts, we completely identify the parametric regions in the open unit cube where $W_{(\alpha,\beta)}(a,x,y)$ is subnormal, hyponormal, semi-hyponormal, and weakly hyponormal. \ As a result, we describe in detail concrete sub-regions where each property holds. \ For instance, we identify the specific sub-region where weak hyponormality holds but semi-hyponormality does not hold, and vice versa. \ To accomplish this, we employ a new technique emanating from the homogeneous orthogonal decomposition of $\ell^2(\mathbb{Z}_+^2)$. \ The technique allows us to reduce the study of semi-hyponormality to positivity considerations of a sequence of $2 \times 2$ scalar matrices. \ It also requires a specific formula for the square root of $2 \times 2$ scalar and operator matrices, and we obtain that along the way. \ As an application of our main results, we show that the Drury-Arveson shift is {\it not} semi-hyponormal. \ Taken together, the new results offer a sharp contrast between the above-mentioned properties for unilateral weighted shifts and their $2$--variable counterparts. 
\end{abstract}

\maketitle

\tableofcontents


\section{\label{Sec0}Introduction}

Let $\mathcal{H}$ be a complex Hilbert space and let $\mathcal{B}(\mathcal{H})$
be the algebra of bounded linear operators acting on $\mathcal{H}$. \ The sets of positive (resp. positive {\it and} invertible)
operators in $\mathcal{B}(\mathcal{H})$ will be denoted by $\mathcal{B}_{+}(\mathcal{H})$ (resp. $\mathcal{B}_{++}(\mathcal{H})$); the adjoint of an operator $T \in \mathcal{B}(\mathcal{H})$ will be denoted by $T^*$. \ An operator $T \in \mathcal{B}(\mathcal{H})$ is said to be {\it normal} if $T^*T = TT^*$, {\it subnormal} if $T$ is the restriction of a normal operator to an invariant subspace, {\it hyponormal} if $T^*T \ge TT^*$ (i.e., the commutator $[T^*,T]$ is in $\mathcal{B}_{+}(\mathcal{H})$), and {\it semi-hyponormal} if $\sqrt{T^*T} - \sqrt{TT^*} \in \mathcal{B}_{+}(\mathcal{H})$. \ (For detailed properties of semi-hyponormality, we refer the reader to \cite{And,Alu,FHM,Xia}.

There are multivariable versions of the above notions, with the most studied being normality, subnormality, and hyponormality for $n$--tuples of commuting operators. \ As expected, a commuting $n$--tuple $\mathbf{T} \equiv (T_1, \ldots, T_n)$ is {\it normal} if for $i=1,\ldots,n$, the operator $T_i$ is normal; {\it subnormal} if $\mathbf{T}$ is the restriction of a normal $n$--tuple to a common invariant subspace; and {\it hyponormal} if the self-commutator of $\mathbf{T}$, defined as $[\mathbf{T}^*,\mathbf{T}] := ([T_j^*,T_i])_{i,j=1}^n$, is positive as an operator acting on $\mathcal{H} \oplus \cdots \oplus \mathcal{H}$ ($n$ copies of $\mathcal{H}$). \ The Lifting Problem for Commuting Subnormals (LPCS) asks whether a subnormal $n$--tuple $\mathbf{T}$, acting on $\mathcal{H}$, always admits a normal extension $\mathbf{N}$, acting on a common Hilbert space $\mathcal{L} \supseteq \mathcal{H}$. \ It is well known that LPCS has been answered in the negative, so efforts have been made to add conditions to guarantee the lifting. \ In \cite[Conjecture 3.3]{CMX}, it was postulated that for a pair $\mathbf{T} \equiv (T_1,T_2)$ of commuting subnormal operators, the hyponormality of $\mathbf{T}$ was sufficient for a lifting. \ In \cite{CuYo1}, we presented examples of hyponormal commuting pairs of subnormal operators on $\ell^2(\mathbb{Z}_+^2)$ without a normal extension. \ Our fundamental class of examples involved a $3$--parametric class of $2$--variable weighted shifts. \ In this paper, we discover a number of new insights into the theory of commuting pairs that satisfy a property {\it weaker} than hyponormality, which will call semi-hyponormality. \ In the process, we will establish significant links with the notion of weak hyponormality for commuting pairs, introduced by J. Conway and W. Szymanski in \cite{ConSzy}. \ 

For unilateral weighted shifts, it is well-known that there is no distinction between hyponormality and semi-hyponormality. \ In sharp contrast with this case, the notions of hyponormality, semi-hyponormality, and weak hyponormality in the bivariate case are all quite different, as will be shown in the sequel. \ Although many of the notions and results in this paper easily generalize to $n$--tuples with $n \ge 3$, we have chosen to stay at the bivariate level, since the properties we wish to study already manifest themselves quite well at $n=2$. 

Let $\mathbf{T} \equiv (T_1,T_2)$ be a commuting pair of bounded operators on $\mathcal{H}$. \ Naturally associated with the pair, there are two $2 \times 2$ operator matrices,
$$
L := \left[
\begin{array}{cc}
T_1^* T_1 & T_2^* T_1 \\
T_1^* T_2 & T_2^* T_2
\end{array}
\right] \quad \quad \textrm{ and } \quad \quad 
R := \left[
\begin{array}{cc}
T_1 T_1^* & T_1 T_2^* \\
T_2 T_1^* & T_2 T_2^*
\end{array}
\right] = 
\left[
\begin{array}{c}
T_1 \\ T_2 
\end{array}
\right] 
\left[
\begin{array}{c}
T_1 \\ T_2 
\end{array}
\right]^*.  
$$
Recall that $\mathbf{T}$ is said to be hyponormal if the matrix $([T_j^*,T_i])_{i,j=1}^2$ (equivalently, $L-R$) is a positive operator; that is, if $L \ge R$. \ Since $R$ is Gramian, $R$ is automatically positive, so the hyponormality of $\mathbf{T}$ implies both the positivity of $L$ and the hyponormality of $T_1$ and $T_2$ as single operators. \ In addition, the hyponormality of $\mathbf{T}$ implies that $\sqrt{L} \ge \sqrt{R}$, by appealing to the L\"owner-Heinz Inequality (Lemma \ref{lemma2}). \ It is also worth mentioning that the transpose of $L$ (as an operator matrix) is a Gramian, that is,
$$
L^t = \left[
\begin{array}{cc}
T_1^* T_1 & T_1^* T_2 \\
T_2^* T_1 & T_2^* T_2
\end{array}
\right] = 
\left[
\begin{array}{cc}
T_1 & T_2 
\end{array}
\right]^* 
\left[
\begin{array}{cc}
T_1 & T_2 
\end{array}
\right].  
$$
Thus, while the positivity of $L$ may fail (as we will see in the sequel), the positivity of its transpose is always guaranteed; thus, Example \ref{Ex1} provides yet another instance of a $2 \times 2$ positive operator matrix whose transpose is {\it not} positive. \ 

Based on the previous observations, and consistent with the classical definition of \linebreak semi-hyponormality for single operators, we now proceed to define semi-hyponormality for pairs of operators, i.e., semi-hyponormality in the bivariate case.

\begin{definition} \label{def11}
Let $\mathbf{T} \equiv (T_1,T_2)$ be a pair of Hilbert space operators, and let $L$ and $R$ be as above. \ 

(i) \ $\mathbf{T}$ is said to be {\it semi-hyponormal} if $L \ge 0$ and $\sqrt{L} \ge \sqrt{R}$. 

(ii) (cf. \cite{ConSzy})$ \ \mathbf{T}$ is said to be {\it weakly hyponormal} if every operator in the linear span of $T_1$ and $T_2$ is hyponormal.
\end{definition}

It is straightforward to verify that:
$$
\textrm{subnormal } \Longrightarrow \textrm{ hyponormal } \Longrightarrow \left\{\begin{array}{c} \textrm{semi-hyponormal} \\
																																																	\textrm{and} \\
																																																	\textrm{weakly hyponormal}
																																								  \end{array}
																													\right.														  
$$
The above diagram gives a perspective of the relative position of properties weaker than subnormality. \ In the opposite direction, one defines $\mathbf{T}$ to be \textit{matricially quasinormal} if
each $T_{i}$ commutes with each $T_{j}^{\ast }T_{k}$, (jointly) \textit{quasinormal} if each $T_{i}$ commutes with each $%
T_{j}^{\ast }T_{j}$, and \textit{spherically quasinormal} if each $T_{i}$
commutes with $\sum_{j=1}^{2}T_{j}^{\ast }T_{j}$, where $i,j,k = 1, 2$. \ As shown in \cite{AtPo} and \cite{Gle}, we have

$$
\begin{array}{ccc}
\textrm{normal} &&\\
\big\Downarrow &&\\
\textrm{matricially quasinormal } &\Longrightarrow &\textrm{ (jointly) quasinormal } \\
&&\big\Downarrow \\
&&\textrm{ spherically quasinormal } \\
&& \big\Downarrow \\
&& \textrm{subnormal}.
\end{array}
$$
On the other hand, the results in \cite{CuYo6} and \cite{Gle} show that the converse implications do not hold. \ However, for $2$--variable weighted shifts, matricially quasinormal and (jointly) quasinormal are equivalent, as we show in Lemma \ref{quasi}. 

In this paper, our emphasis is on semi-hyponormality and weak hyponormality, and how they relate to hyponormality and subnormality. \ We will first derive some basic structural properties of semi-hyponormality for commuting pairs of operators; this will include an analysis of positivity for $2 \times 2$ operator matrices, the calculation of their square roots, and criteria for the commutativity of their operator entries. \ Next, we will focus on the case of bivariate weighted shifts, and study in detail the fundamental class of bivariate weighted shifts determined by the three parameters $a,x,y$ in the open unit interval $(0,1)$; that is, $\mathbf{T}$ is the bivariate weighted shift $W_{(\alpha,\beta)}(a,x,y)$ defined by
$$
\alpha_{(0,0)}:=x, \; \alpha_{(0,k_2)}:=a \,\, (\textrm{for all }k_2 \ge 1), \; \alpha_{(k_1,k_2)}:=1 \, \, (\textrm{for all }k_1 \ge 1, k_2 \ge 0);
$$
$$
\beta_{(0,0)}:=y, \; \beta_{(k_1,0)}:=\dfrac{ay}{x} \,\, (\textrm{for all }k_1 \ge 1), \; \beta_{(k_1,k_2)}:=1 \, \, (\textrm{for all }k_1 \ge 1, k_2 \ge 0). \label{defaxy}
$$
The weight diagram for $W_{(\alpha,\beta)}(a,x,y)$ is shown in Figure \ref{Figure 1}(ii) on page \pageref{Figure 1}.

The class $\mathcal{W}:=\{W_{(\alpha,\beta)}(a,x,y): a,x,y \in (0,1)^3 \textrm{ and } ay<x \}$ is the simplest nontrivial family of $2$--variable weighted shift that cannot be represented as a tensor product $(I \otimes W_\omega,W_\eta \otimes I)$ (for which the notions we discuss here can be derived from those of $W_\omega$ and $W_\eta$). \ In short, if we wish to stay away from trivial doubly commuting pairs, the first nontrivial instance is provided by the class $\mathcal{W}$. 

To exhibit concretely the above-mentioned differences, we will partition the open unit cube $(0,1)^3$ of triples $(a,x,y)$ into regions characterizing the subnormality, hyponormality, \linebreak semi-hyponormality, and weak subnormality of the above-mentioned bivariate weighted shift $W_{(\alpha,\beta)}(a,x,y)$. \ This will allow us to describe in detail concrete sub-regions, where, for instance, weak hyponormality holds but semi-hyponormality does not hold, and vice versa (see Figure 7). \ As an application of our main results, we show that the Drury-Arveson shift is {\it not} semi-hyponormal. \ These new results for $2$--variable weighted shifts offer a revealing contrast with semi-hyponormality in one variable. 

In view of the definitions of hyponormality and semi-hyponormality, knowing when a $2 \times 2$ operator matrix is positive, is of great importance and significance. \ In 1959, J.L. Smul'jan \cite{Smu} proved that a $2 \times 2$ operator matrix
$\left(
\begin{array}{ll}
A & B \\
B^{\ast } & C%
\end{array}%
\right)$ 
is positive if and only if $A,C \ge 0$, and $B=\sqrt{A}E\sqrt{C}$ for some contraction $E$ (cf. \cite[Proposition 2.1]{CF}). \ More recently, M.S. Moslehian, M. Kian and Q. Xu \cite{MKX} obtained a characterization of the positivity of $2\times 2$ block matrices of operators on Hilbert space, in case $A \in \mathcal{B}_{++}(\mathcal{H})$. \ That is, they generalize Choleski's Algorithm for matrices 
$M:=\left(
\begin{array}{ll}
A & B \\
B^{\ast } & C%
\end{array}%
\right)
\in \mathcal{B}(\mathcal{H}\oplus \mathcal{H})$ with $B,C\in
\mathcal{B}(\mathcal{H})$ and $A\in \mathcal{B}_{++}(\mathcal{H})$. \ They proved that $M$ is
positive if and only if $C\geq B^{\ast }A^{-1}B$. \ When $\dim \mathcal{H}<\infty$, $M$ is positive if and only if $C\geq B^{\ast }A^{\dag }B$, where $A^{\dag }$ is the Moore-Penrose inverse of $A$. \ 

For an orthogonally
diagonalizable $n \times n$ matrix $A$ with $A=UDU^*$, where $U$ is a unitary and $D$ is diagonal, it is well known that $A^{\frac{1}{2}}=UD^{\frac{1}{2}}U^*$, and $U$ can be chosen to be the matrix of eigenvectors whose eigenvalues are the diagonal entries of $D$. \ Our first task in this paper is to use an appropriate version of the Cayley-Hamilton Theorem to find an answer to the following question.

\begin{problem}
(\cite[Remark 5.2]{MKX}; see also \cite{MO}) \label{Prob1} Find out an explicit formula for $M^{%
\frac{1}{2}}$ under certain restrictions on $X$ and $B$, where $M=\left(
\begin{array}{ll}
A & X \\
X^{\ast } & B
\end{array}
\right) \in \mathcal{B}(\mathcal{H}\oplus \mathcal{H})$, $B,X\in \mathcal{B}(%
\mathcal{H})$, and $A\in \mathcal{B}_{++}(\mathcal{H})$.
\end{problem}

Problem \ref{Prob1} is related to the Kato Square Root Problem, formulated by T. Kato in 1953 \cite{Kato}, and solved in the affirmative in 2002 for the case of second order elliptic operators on $\mathbb{R}^n$ \cite{AHLMT}; that is, given an $n \times n$ matrix of $L^\infty$ coefficients and a second order divergence form operator $Lf := -\operatorname{div}(A \nabla f)$, the domain of $\sqrt{L}$ coincides with the Sobolev space $H^1(\mathbb{R}^n)$ and $\left\|\sqrt{L}f\right\|_2 \le \left\|\nabla f\right\|_2$ for every $f \in H^1(\mathbb{R}^n)$. \ While for unbounded operators the solution includes the matter of domain for the square root, in the case of bounded operators we will focus primarily on an algebraic solution as explicit as possible. \ (For more information on square roots of classical bounded operators, the reader is referred to \cite{MPR}.) 

The solution of Problem \ref{Prob1} will provide a new tool to study the differences between hyponormality and semi-hyponormality. \ To solve Problem \ref{Prob1}, we will employ a new technique which exploits the existence of a chain of orthogonal finite dimensional subspaces of $\ell^2(\mathbb{Z}_+^2)$ of increasing dimension, denoted by $\{\mathcal{K}(n)\}_{n \ge 0}$; the formal definition of $\mathcal{K}(n)$ is given in Subsection \ref{subsection25}. \ This matricial representation of $2$--variable weighted shifts is based on the decomposition of $\ell^2(\mathbb{Z}_+^2)$ along subspaces $\mathcal{K}(n)$ determined by cross diagonals drawn through the lattice points of 
$\mathbb{Z}_+^2$, and organized according to the degree-lexicographic order. \ These subspaces are reducing for $[T_{i}^{\ast },T_{j}]$ and $[T_{i},T_{j}^{\ast }]$ (for all $i,j \ge 0$). \ As a result, the calculation of the square roots of $L$ and $R$ (needed to determine semi-hyponormality) can therefore be carried out at the level of $2 \times 2$ scalar matrices. \ For the above-mentioned class $\mathcal{W} =\{W_{(\alpha,\beta)}(a,x,y): 0<a,x,y<1 \textrm{ and } ay<x\}$, this new matricial representation will yield $2 \times 2$ matrices for which square roots can be easily computed.

We alert the reader to the existence of two other multivariate notions of semi-hyponormality. \ The first one was introduced by D. Xia \cite{Xia} in 1983, and it uses a commuting $n$--tuple $\mathbf{U}$ of unitary operators and a single positive operator $A$ ; the $(n+1)$--tuple $(\mathbf{U},A)$ is called semi-hyponormal when a certain operator inequality involving $A$ and $\mathbf{U}$ holds. \ In the special case when $n=1$, the pair $(A,U)$ is semi-hyponormal if and only if the operator $UA$ is semi-hyponormal. \ D. Xia's notion is quite different from ours (even in the case of unilateral weighted shifts, since in the polar decomposition of such shifts the partial isometric factor does not admit a unitary extension), and tailor-made to the requirements needed for the construction of a singular integral model for semi-hyponormality in several variables. \ The second notion is based on the spherical polar decomposition of a commuting pair \cite{KKY,KKY2}; that is, if $\left(
\begin{array}{c}
T_{1} \\
T_{2}
\end{array}
\right) =V\left\vert T\right\vert$, where $\left\vert T\right\vert:=\sqrt{T_1^*T_1+T_2^*T_2}$ and $V \equiv \left(
\begin{array}{c}
T_{1} \\
T_{2}
\end{array}
\right)$ is a joint partial isometry, one can compare the positive parts of $\mathbf{T}$ and $\mathbf{T}^*$ and define spherical $p$--hyponormality, as follows: for $0<p<1$, the commuting pair $\mathbf{T}$ is spherically $p$--hyponormal if $\left\vert \mathbf{T}\right\vert \ge \left\vert \mathbf{T}^*\right\vert$ \cite{KKY,KKY2}. \ Although this notion has interesting connections with the spherical Aluthge transform, it does not relate well to the matrices $L$ and $R$ which form the essential foundation for Definition \ref{def11} and the topics in this paper.


\section{Notation and Preliminary Results}


\subsection{Some Auxiliary Results}

For the reader's convenience, in this subsection we gather several well-known auxiliary results which are needed for the proofs of the main results in this article.

\begin{lemma}
\cite{Fur} \label{lemma2} (L\"{o}wner-Heinz Inequality) \ Let $\mathcal{H}$
be a complex Hilbert space, and let \linebreak $A,B\in \mathcal{B}(\mathcal{H})$ such
that $0\leq B\leq A$, then $0\leq B^{\alpha }\leq A^{\alpha }$ for any $%
\alpha \in \left[ 0,1\right] $.
\end{lemma}

\begin{lemma} \ (Smul'jan's Lemma \cite{Smu}) \label{lemma1} \ On $\mathcal{H}_1 \oplus \mathcal{H}_2$, consider the $2 \times 2$ operator matrix 
\begin{equation*}
M=\left(
\begin{array}{ll}
A & X \\
X^* & B
\end{array}
\right).
\end{equation*}
\ Then $M\geq 0$ if and only if\newline
(i) \ $A, B \geq 0$; and
\newline
(ii) \ $X=\sqrt{A}E \sqrt{B}$, for some contractive linear operator $E:\mathcal{H}_{2}\longrightarrow \mathcal{H}_{1}$.
\end{lemma}



\subsection{$2$--Variable Weighted Shifts}

We first recall the definition of a unilateral weighted shift. \ For $\alpha \equiv \{\alpha _{n}\}_{n=0}^{\infty }$ a bounded sequence of positive real numbers (called {\it weights}), let \linebreak $W_{\alpha}\equiv \mathrm{shift}(\alpha _{0},\alpha _{1},\ldots ):\ell ^{2}(\mathbb{Z}_{+})\rightarrow \ell ^{2}(\mathbb{Z}_{+})$ be the associated unilateral weighted shift, defined by $W_{\alpha }e_{n}:=\alpha _{n}e_{n+1}\;$(all $n\geq 0$), where $\{e_{n}\}_{n=0}^{\infty }$ is the canonical orthonormal
basis in $\ell ^{2}(\mathbb{Z}_{+})$. 

Consider now two double-indexed positive bounded sequences $\alpha _{%
\mathbf{k}},\beta _{\mathbf{k}}\in \ell ^{\infty }(\mathbb{Z}_{+}^{2})$ , $%
\mathbf{k}\equiv (k_{1},k_{2})\in \mathbb{Z}_{+}^{2}:=\mathbb{Z}_{+}\times 
\mathbb{Z}_{+}$ and let $\ell ^{2}(\mathbb{Z}_{+}^{2})$\ be the Hilbert
space of square-summable complex sequences indexed by $\mathbb{Z}_{+}^{2}$.
\ (Recall that $\ell ^{2}(\mathbb{Z}_{+}^{2})$ is canonically isometrically
isomorphic to $\ell ^{2}(\mathbb{Z}_{+})\bigotimes \ell ^{2}(\mathbb{Z}_{+})$%
.) \ We define the $2$--variable weighted shift $\mathbf{T}\equiv
(T_{1},T_{2})$\ by 
\begin{equation*}
T_{1}e_{\mathbf{k}}:=\alpha _{\mathbf{k}}e_{\mathbf{k+}\varepsilon _{1}}
\end{equation*}%
\begin{equation*}
T_{2}e_{\mathbf{k}}:=\beta _{\mathbf{k}}e_{\mathbf{k+}\varepsilon _{2}},
\end{equation*}%
where $\mathbf{\varepsilon }_{1}:=(1,0)$ and $\mathbf{\varepsilon }%
_{2}:=(0,1)$. \ Clearly, 
\begin{equation}
T_{1}T_{2}=T_{2}T_{1}\Longleftrightarrow \beta _{\mathbf{k+}\varepsilon
_{1}}\alpha _{\mathbf{k}}=\alpha _{\mathbf{k+}\varepsilon _{2}}\beta _{%
\mathbf{k}}\;\;(\text{all }\mathbf{k}).  \label{commuting}
\end{equation}
Associated to a $2$--variable weighted shift is a $2$--dimensional weight diagram (see Figure \ref{Figure 1}(i)). \ 


\setlength{\unitlength}{1mm} \psset{unit=1mm}
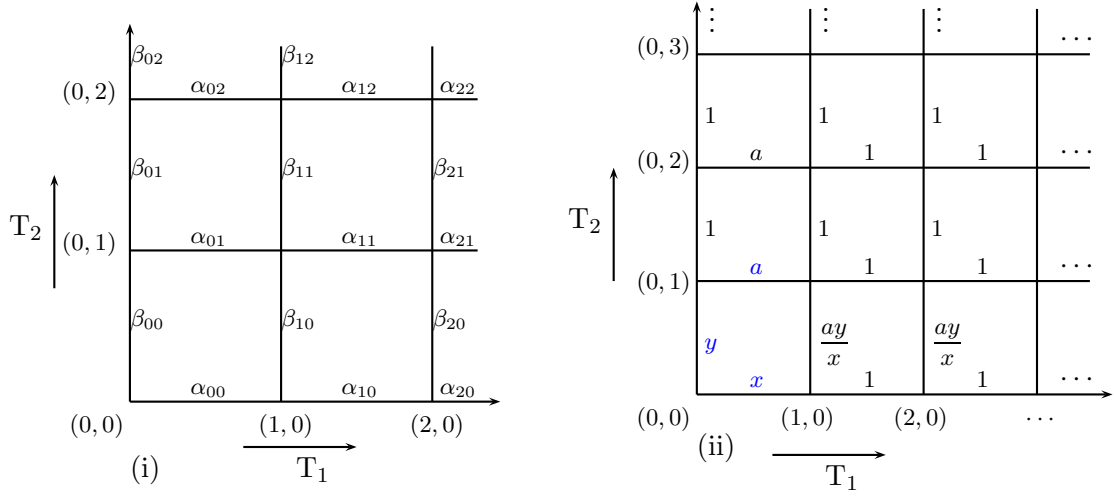
\begin{figure}[th]
\begin{center}
\begin{picture}(135,65)

\rput(-15,25){
\psline{->}(20,20)(69,20)
\psline(20,40)(66,40)
\psline(20,60)(66,60)
\psline{->}(20,20)(20,70)
\psline(40,20)(40,67)
\psline(60,20)(60,67)

\put(12,16){\footnotesize{$(0,0)$}}
\put(37,16){\footnotesize{$(1,0)$}}
\put(57,16){\footnotesize{$(2,0)$}}

\put(28,21){\footnotesize{$\alpha_{00}$}}
\put(48,21){\footnotesize{$\alpha_{10}$}}
\put(61,21){\footnotesize{$\alpha_{20}$}}

\put(28,41){\footnotesize{$\alpha_{01}$}}
\put(48,41){\footnotesize{$\alpha_{11}$}}
\put(61,41){\footnotesize{$\alpha_{21}$}}

\put(28,61){\footnotesize{$\alpha_{02}$}}
\put(48,61){\footnotesize{$\alpha_{12}$}}
\put(61,61){\footnotesize{$\alpha_{22}$}}

\psline{->}(35,14)(50,14)
\put(42,10){$\rm{T}_1$}
\psline{->}(10,35)(10,50)
\put(4,42){$\rm{T}_2$}

\put(11,40){\footnotesize{$(0,1)$}}
\put(11,60){\footnotesize{$(0,2)$}}

\put(20,30){\footnotesize{$\beta_{00}$}}
\put(20,50){\footnotesize{$\beta_{01}$}}
\put(20,65){\footnotesize{$\beta_{02}$}}

\put(40,30){\footnotesize{$\beta_{10}$}}
\put(40,50){\footnotesize{$\beta_{11}$}}
\put(40,65){\footnotesize{$\beta_{12}$}}

\put(60,30){\footnotesize{$\beta_{20}$}}
\put(60,50){\footnotesize{$\beta_{21}$}}

\put(20,10){(i)}
}

\rput(5,29){

\psline{->}(75,20)(130,20)
\psline(75,35)(127,35)
\psline(75,50)(127,50)
\psline(75,65)(127,65)

\psline{->}(75,20)(75,72)
\psline(90,20)(90,71)
\psline(105,20)(105,71)
\psline(120,20)(120,71)

\put(82,21){\footnotesize{$\cbl{x}$}}
\put(97,21){\footnotesize{$1$}}
\put(112,21){\footnotesize{$1$}}

\put(82,36){\footnotesize{$\cbl{a}$}}
\put(97,36){\footnotesize{$1$}}
\put(112,36){\footnotesize{$1$}}

\put(82,51){\footnotesize{$a$}}
\put(97,51){\footnotesize{$1$}}
\put(112,51){\footnotesize{$1$}}


\put(76,26){\footnotesize{$\cbl{y}$}}
\put(76,41){\footnotesize{$1$}}
\put(76,56){\footnotesize{$1$}}

\put(91,26){\footnotesize{$\dfrac{ay}{x}$}}
\put(91,41){\footnotesize{$1$}}
\put(91,56){\footnotesize{$1$}}

\put(106,26){\footnotesize{$\dfrac{ay}{x}$}}
\put(106,41){\footnotesize{$1$}}
\put(106,56){\footnotesize{$1$}}

\put(91.5,68){$\vdots$}
\put(76.5,68){$\vdots$}
\put(106.5,68){$\vdots$}

\put(123,21.1){$\cdots$}
\put(123,36.1){$\cdots$}
\put(123,51.1){$\cdots$}
\put(123,66.1){$\cdots$}

\psline{->}(85,12)(100,12)
\put(92,8){$\rm{T}_1$}
\psline{->}(64,35)(64,50)
\put(58,42){$\rm{T}_2$}

\put(67,16){\footnotesize{$(0,0)$}}
\put(86,16){\footnotesize{$(1,0)$}}
\put(101,16){\footnotesize{$(2,0)$}}
\put(118.5,16){\footnotesize{$\cdots$}}

\put(67,34){\footnotesize{$(0,1)$}}
\put(67,50){\footnotesize{$(0,2)$}}
\put(67,65){\footnotesize{$(0,3)$}}


\put(75,12){(ii)}

}
\end{picture}
\end{center}
\caption{(i) Weight diagram of a generic $2$--variable weighted shift $W_{(\alpha,\beta)}~\equiv~(T_{1},T_{2})$; (ii) weight diagram of 
$W_{(\alpha,\beta)}(a,x,y)$,~defined on page \pageref{defaxy}.}

\label{Figure 1}
\end{figure}


Trivially, a pair of unilateral weighted shifts $W_{\alpha}$ and $W_{\beta }$
gives rise to a $2$--variable weighted shift $\mathbf{T}\equiv (T_{1},T_{2})$, if we let $\alpha _{(k_{1},k_{2})}:=\alpha _{k_{1}}$ and $\beta
_{(k_{1},k_{2})}:=\beta _{k_{2}}\;$(all $k_{1},k_{2}\in \mathbb{Z}_{+}^{2}$%
). \ In this case, $\mathbf{T}$ is subnormal (resp. hyponormal) if and only
if so are $T_{1}$ and $T_{2}$; in fact, under the canonical identification
of $\ell ^{2}(\mathbb{Z}_{+}^{2})$ and $\ell ^{2}(\mathbb{Z}_{+})\bigotimes
\ell ^{2}(\mathbb{Z}_{+})$, $T_{1}\cong I \bigotimes W_{a}$ and $T_{2}\cong
W_{\beta } \bigotimes I$, and $\mathbf{T}$ is also doubly commuting. \ For
this reason, we do not focus attention on shifts of this type, and use them
only when the above-mentioned triviality is desirable or needed. \ The special case when $W_\alpha = W_\beta = U_+$ (the un-weighted unilateral shift), the $2$--variable weighted shift $(I \otimes U_+,U_+ \otimes I)$ is known as the Helton-Howe shift ($\operatorname{HHS}$).

We also recall the notion of moment of order $\mathbf{k} \equiv (k_1,k_2)$ for a pair $%
(\alpha ,\beta )$ satisfying (\ref{commuting}). \ Given $\mathbf{k}\in 
\mathbb{Z}_{+}^{2}$, the moment of $(\alpha ,\beta )$ of order $\mathbf{k}$
is 
\begin{equation} \label{moment1}
\gamma _{\mathbf{k}}\equiv \gamma _{\mathbf{k}}(\alpha ,\beta ):=\left\{ 
\begin{array}{lr}
1 & \text{if }\mathbf{k}=\mathbf{0} \\ 
\alpha _{(0,0)}^{2}\cdot ...\cdot \alpha _{(k_{1}-1,0)}^{2} & \text{if } k_{1}\geq 1\text{ and }k_{2}=0 \\ 
\beta _{(0,0)}^{2}\cdot ...\cdot \beta _{(0,k_{2}-1)}^{2} & \text{if }k_{1}=0 \text{ and }k_{2}\geq 1 \\ 
\alpha _{(0,0)}^{2}\cdot ...\cdot \alpha _{(k_{1}-1,0)}^{2}\cdot \beta_{(k_{1},0)}^{2}\cdot ...\cdot \beta _{(k_{1},k_{2}-1)}^{2} & \text{if } k_{1}\geq 1\text{ and }k_{2}\geq 1.
\end{array}%
\right.
\end{equation}
We remark that, due to (\ref{commuting}), $\gamma _{\mathbf{k}}$ can be computed using any non-decreasing path from $%
(0,0)$ to $(k_{1},k_{2})$.

The natural analog of the Berger-Gellar-Wallen Theorem for unilateral weighted shifts \linebreak (cf. \cite[III.8.16]{Con}) is true, as we now state.

\begin{theorem}
\label{Berger}(The Berger-Gellar-Wallen Theorem in two variables \cite{JeLu}) \ A $2$--variable weighted shift $\mathbf{T}\equiv (T_{1},T_{2})$ admits a commuting normal extension if and only if there is a probability measure $\mu$ defined on the $2$-dimensional rectangle $R=[0,a_{1}]\times \lbrack 0,a_{2}]$
($a_{i}:=\left\| T_{i}\right\| ^{2}$) such that $\gamma _{\mathbf{k}%
}=\iint_{R}\mathbf{t}^{\mathbf{k}}d\mu (\mathbf{t}):=%
\iint_{R}t_{1}^{k_{1}}t_{2}^{k_{2}}\;d\mu (t_{1},t_{2})$ \ $($all $\mathbf{k\in }\;\mathbb{Z}_{+}^{2}$).
\end{theorem}

To state the next result, we first need to recall the definition of the Helton-Howe shift (HHS), acting on $\ell^2(\mathbb{Z}_+^2)$; using tensor notation, this $2$--variable weighted shift is defined as $\operatorname{HHS} := (I \otimes U_+,U_+ \otimes I)$, that is, all entries in the weight diagram of $\operatorname{HHS}$ are equal to $1$. 

\begin{lemma} \label{quasi}
Let $W_{(\alpha ,\beta)}\equiv
(T_{1},T_{2})$ be a commuting $2$--variable weighted shift. \ The following statements are equivalent:%
\newline
(i) \ $W_{\left( \alpha ,\beta \right) }$ is matricially quasinormal, i.e., each $T_{i}$ commutes with each $T_{j}^{\ast }T_{k}$ \quad ($i,j,k=1,2$);
(ii) $\ W_{\left( \alpha ,\beta \right) }$ is jointly quasinormal, i.e., each $T_{i}$ commutes with each $T_{j}^{\ast
}T_{j}$ \quad ($i,j=1,2$);\newline
(iii) \ Up to a constant multiple, $W_{\left( \alpha ,\beta \right) }$ is the
Helton-Howe shift.
\end{lemma}

\begin{proof} \ (i) $\Rightarrow (ii)$ is clear from the definitions, and (ii) $\Rightarrow$ (iii) is the content of \cite[Theorem 3.3]{CLY3}. \ Therefore, and without loss of generality, we only need to prove that the Helton-Howe shift is matricially quasinormal. \ Since $\operatorname{HHS} =(I \otimes U_+,U_+ \otimes I)$, the result follows from a straightforward calculation with tensors. \ 
\end{proof} 


Next, we state three results about (joint) hyponormality and $k$--hyponormality of $2$--variable weighted shifts that will be needed to prove our main results.

\begin{lemma}
\label{joint hypo} (Six-point Test \cite{bridge}) \ A $2$--variable
weighted shift $W_{\left( \alpha ,\beta \right) }$ is jointly hyponormal if and only if%
\begin{equation*}
\left(
\begin{array}{cc}
\alpha_{\mathbf{k+\varepsilon_1}}^2 - \alpha_{\mathbf{k}}^{2} & \alpha
_{\mathbf{k+\varepsilon_2}}\beta_{\mathbf{k+\varepsilon_1}}-\alpha_{\mathbf{k}}\beta
_{\mathbf{k}} \\
\alpha
_{\mathbf{k+\varepsilon_2}}\beta _{\mathbf{k+\varepsilon_1}}-\alpha _{\mathbf{k}}\beta
_{\mathbf{k}} & \beta_{\mathbf{k+\varepsilon_2}}^2 - \beta_{\mathbf{k}}^{2}
\end{array}%
\right) \geq 0
\end{equation*}
for all $\mathbf{k}\in \mathbf{Z}_{+}^{2}$.
\end{lemma}

\begin{lemma}
\label{khypo} (Joint hyponormality of $2$--variable weighted shifts \cite{CLY1}) \ Let $W_{(\alpha ,\beta )}$ be a 2--variable
weighted shift. \ The following statements are equivalent:\newline
(i) \ $W_{(\alpha ,\beta )}$ is hyponormal;\newline
(ii) \ $M_{\mathbf{u}}( W_{(\alpha ,\beta )}) :=\left(
\begin{array}{ccc}
\gamma_{\mathbf{u}} & \gamma _{\mathbf{u+\varepsilon_1}} & \gamma
_{\mathbf{u+\varepsilon_2}} \\
\gamma _{\mathbf{u+\varepsilon_1}} & \gamma _{\mathbf{u}+2\varepsilon_1} & \gamma
_{\mathbf{u+\varepsilon_1+\varepsilon_2}} \\
\gamma _{\mathbf{u+\varepsilon_2}} & \gamma _{\mathbf{u+\varepsilon_1+\varepsilon_2}} & \gamma
_{\mathbf{u}+2\varepsilon_2}
\end{array}
\right) \geq 0 \quad \quad (\textrm{all } \mathbf{u} \in \mathbb{Z}_{+}^2$).
\end{lemma}

\begin{lemma} \label{khypon} ($k$--hyponormality of $2$--variable weighted shifts \cite[Theorem 2.4]{CLY1}) \ Let $\mathbf{T} \equiv\left(T_1, T_2\right)$ be a 2-variable weighted shift with weight sequences $\alpha \equiv\left\{\alpha_{\mathbf{k}}\right\}$ and $\beta \equiv\left\{\beta_{\mathbf{k}}\right\}$. The following statements are equivalent. \newline
(a) \ $\mathbf{T}$ is k-hyponormal. \newline
(b) \ $M_{\mathbf{u}}(k):=\left(\gamma_{\mathbf{u}+(m, n)+(p, q)}\right)_{\substack{0 \leqslant m+n \leqslant k \\ 0 \leqslant p+q \leqslant k}} \geqslant 0 \textrm{ for all } \mathbf{u} \in \mathbb{Z}_{+}^2$.
\end{lemma}

\vspace{10pt}

While tensoring two unilateral weighted shifts does not produce interesting examples of $2$--variable weighted shift, the following construction does yield highly nontrivial examples.

\begin{definition} \label{def26} \ (Canonical embedding of two unilateral shifts) \ In $\ell^2(\mathbb{Z}_+)$, consider two unilateral weighted shifts $W_\omega \equiv \operatorname{shift}(\omega_0,\omega_1,\ldots)$ and $W_\eta \equiv \operatorname{shift}(\eta_0,\eta_1,\ldots)$ given by two bounded sequences of positive numbers $\omega$ and $\eta$, resp. \ Let $\alpha,\beta \in \ell^\infty(\mathbb{Z}_+^2)$ be defined as $\alpha_{(k_1,k_2)}:=\omega_{k_1+k_2}$ and $\beta_{(k_1,k_2)}:=\eta_{k_1+k_2}$, for all $\mathbf{k} \in \mathbb{Z}_+^2)$. \ The shift $W_{(\alpha,\beta)}$ is called the canonical embedding of $(W_\omega,W_\eta)$.
\end{definition}

\begin{remark} \label{rem27}
(i) \ Observe that a canonical embedding of $(W_\omega,W_\eta)$ is commutative if and only if 
$$
\omega_\ell \cdot \eta_{\ell+1} = \omega_{\ell+1} \cdot \eta_\ell,
$$
for all $\ell \ge 0$.

(ii) \ In \cite{CR2}, the notion of canonical embedding of a unilateral weighted shift into a $2$--variable weighted shift was introduced. \ Subsequently, in \cite{CLY4}, polynomial embeddings of unilateral weighted shifts were defined. \ Although there are conceptual similarities of these two notions with Definition \ref{def26}, the new notion of canonical embedding provides a more robust and general construction of how {\it two} unilateral shifts can generate, in a natural way, a $2$--variable weighted shift. \ (For a detailed discussion of $2$--variable weighted shifts, the reader is referred to \cite{CLY1,CLY2,CuYo1,CuYo2,CuYo6,CuYo7}.)
\end{remark}


\subsection{The Square Root of a $2 \times 2$ Scalar Matrix}

We briefly discuss a specific formula for the square root of a nonzero $2 \times 2$ matrix over the complex numbers. \ Consider a positive semi-definite $2 \times 2$ matrix $M$, with eigenvalues $\lambda_1$ and $\lambda_2$, and let $\tr$ and $\det$ denote trace and determinant, resp. \ By the Cayley-Hamilton Theorem, we know that $M$ satisfies the characteristic polynomial equation $z^2 - (\tr M) z + \det M = 0$; that is, we have
$$
M^2 - \tr M \cdot M + \det M \cdot I = 0.
$$
Since $M \ge 0$, it has a square root $\sqrt{M} \ge 0$, which satisfies its own characteristic polynomial equation:
\begin{equation} \label{squareroot}
(\sqrt{M})^2 - (\tr \sqrt{M}) \cdot \sqrt{M} + (\det \sqrt{M}) \cdot I = 0.
\end{equation}
We know that $\det \sqrt{M} = \sqrt{\det M}$; to find the trace of $\sqrt{M}$, we need a little algebra:
$$
\tr \sqrt{M} = \sqrt{\lambda_1} + \sqrt{\lambda_2} = \sqrt{\lambda_1 + \lambda_2 + 2 \sqrt{\lambda_1}\sqrt{\lambda_2}} =
\sqrt{\tr M + 2 \sqrt{\det M}}. 
$$ 
From (\ref{squareroot}) we obtain
$$
M - (\sqrt{\tr M + 2 \sqrt{\det M}}) \cdot \sqrt{M} + (\sqrt{\det M}) \cdot I = 0.
$$
If we now solve for $\sqrt{M}$, we obtain
\begin{equation} \label{squareroot2}
\sqrt{M} = \dfrac{M + (\sqrt{\det M}) \cdot I}{\sqrt{\tr M + 2 \sqrt{\det M}}}.
\end{equation}
Unlike other formulas in the literature (which use a combination of quantities related to $M$ and $\sqrt{M}$), the formula (\ref{squareroot2}) expresses $\sqrt{M}$ entirely in terms of $M$, its trace and its determinant. \ Also, in the special case of a singular matrix $M$, we recover the well-known expression $\sqrt{M}=\dfrac{M}{\sqrt{\tr M}}$.

\bigskip


\subsection{The Square Root of a $2 \times 2$ Operator Matrix} \label{commL} 

As we know, the Cayley-Hamilton Theorem holds in any unital commutative ring. \ The formula (\ref{squareroot2}), however, needs a notion of adjoint and the invertibility of the trace of a nonzero positive matrix. \ Thus, if we restrict attention to $2 \times 2$ operator matrices with entries in an abelian unital $C^*$--subalgebra of $\mathcal{B}(\mathcal{H})$, (\ref{squareroot2}) still holds. \ In this section, our goal is to obtain a concrete formula for the square root of a positive $2 \times 2$ operator matrix with commuting entries and, in the process, to give a solution to Problem \ref{Prob1}.

Given a nonzero matrix 
$$
M=\left(
\begin{array}{ll}
A & X \\
X^{\ast } & B
\end{array}
\right) \in \mathcal{B}(\mathcal{H} \oplus \mathcal{H}), 
$$
with commuting entries, let $\mathcal{C}$ be the unital $C^*$--subalgebra of $\mathcal{B}(\mathcal{H})$ generated by $A$, $X$, and $B$. \ Now consider the algebra $M_2(\mathcal{C})$ of $2 \times 2$ matrices with entries from $\mathcal{C}$. \ It is well known that $M_2(\mathcal{C})$ can be made into a $C^*$--algebra, if we define the norm of a $2 \times 2$ matrix $M$ as the operator norm of $M$ acting on $\mathcal{B}(\mathcal{H} \oplus \mathcal{H})$. \ In this setting we can now state a concrete formula for the square root of a positive matrix $M$ with commuting entries satisfying the conditions in Problem \ref{Prob1}. \ As in the scalar case, we let $\tr$ and $\det$ denote the trace and the determinant of a matrix; that is, $\tr M := A+B$ and $\det M := AB-X^*X$. \ Under the assumption of positivity for $M$, we know that $\sqrt{M}$ exists and is a positive element of $M_2(\mathcal{C})$. \ In addition, we also know that both $\tr M$ and $\det M$ are positive operators, and therefore they have square roots which belong to $\mathcal{C}$. \ Finally, observe that $\tr M$ is a positive and invertible operator in $\mathcal{C}$, since it dominates $A \in \mathcal{B}_{++}(\mathcal{H})$.

\begin{theorem} \label{Thm1} \ Assume that the matrix 
$$
M=\left(
\begin{array}{ll}
A & X \\
X^{\ast } & B
\end{array}
\right) \in \mathcal{B}(\mathcal{H} \oplus \mathcal{H}), 
$$
is positive, has commuting entries, and $A \in \mathcal{B}_{++}(\mathcal{H})$. \ Then
\begin{equation*}
\sqrt{M}=\left(
\begin{array}{ll}
\left( \tr M+2Q\right) ^{-\frac{1}{2}} & 0 \\
0 & \left( \tr M+2Q\right) ^{-\frac{1}{2}}%
\end{array}%
\right) \left(
\begin{array}{ll}
A+Q & X \\
X^{\ast } & B+Q%
\end{array}%
\right) .
\end{equation*}
\end{theorem}

\begin{proof}
This is straightforward from the formula in (\ref{squareroot2}), appropriately adjusted to the case of operator matrices in $M_2(\mathcal{C})$.
\end{proof}

Next, we consider the special case when $\det M=0$. \ The following result is a direct application of Theorem \ref{Thm1}.

\begin{corollary}
\label{Cor1} Let $M\in \mathcal{B}(\mathcal{H}\oplus \mathcal{H})$ be as in
Theorem \ref{Thm1}, and assume that $\det M=0$. \ Then
\begin{equation*}
\sqrt{M}=\left(
\begin{array}{ll}
\left( \tr M\right) ^{-\frac{1}{2}} & 0 \\
0 & \left( \mathrm{Tr}M\right) ^{-\frac{1}{2}}%
\end{array}%
\right) \left(
\begin{array}{ll}
A & X \\
X^{\ast } & B%
\end{array}%
\right) .
\end{equation*}
\end{corollary}

\medskip
Our next result requires only that $AX=XA$.

\begin{proposition}
\label{Prop1} Let $M=\left(
\begin{array}{ll}
A & X \\
X^{\ast } & B%
\end{array}%
\right) \in \mathcal{B}(\mathcal{H}\oplus \mathcal{H})$, $B,X\in \mathcal{B}(%
\mathcal{H})$, and $A\in \mathcal{B}_{++}(\mathcal{H})$. \ Assume that $\det M = 0$ and that $AX=XA$. \ Then $M$ is a {\it flat extension} of $A$, i.e.,
$$
M=\left(
\begin{array}{ll}
A & AW \\
W^{\ast }A & W^{\ast }AW
\end{array}
\right)
$$
for some $W\in \mathcal{B}(\mathcal{H})$; that is, $X=AW$ and $B=W^*AW$.
\end{proposition}

\begin{proof}
Since $A$ is invertible, Lemma 2.2 readily implies that
$$
M=\left(
\begin{array}{ll}
A & AW \\
W^{\ast }A & W^{\ast }AW%
\end{array}%
\right) +\left(
\begin{array}{ll}
0 & 0 \\
0 & C
\end{array}\right)
=
\left(
\begin{array}{ll}
A & AW \\
W^{\ast }A & W^{\ast }AW + C
\end{array}%
\right)
$$
for some $C,W\in \mathcal{B}(\mathcal{H})$ with $C\geq 0$. \ Then
\begin{eqnarray*}
\det M &=&AW^{\ast }AW+AC-W^{\ast }AAW=AX^{\ast }W-X^{\ast }X+AC \\
&=&X^{\ast }AW-X^{\ast }X+AC=X^{\ast }X-X^{\ast }X+AC=AC.
\end{eqnarray*}%
Since $\det M=0$, it follows that $AC=0$, and therefore $C=0$. \ Then $M$ is a flat extension of $A$, as desired.
\end{proof}


\section{Main Results} \label{Sec3}

\subsection{The Homogeneous Orthogonal Decomposition of $\ell^2(\mathbb{Z}_+^2)$} \label{subsection25}

The Hilbert space $\ell^2(\mathbb{Z}_+^2)$ admits a decomposition as an orthogonal direct sum of finite dimensional subspaces of increasing dimension. \ That is, for $n \ge 0$, let
$$
\mathcal{K}(n) := \bigvee \{e_{\mathbf{k}}: \mathbf{k} \equiv (k_1,k_2) \in \mathbb{Z}_+^2 \textrm{ and } k_1+k_2 = n \}.
$$
Then 
\begin{equation} \label{homog}
\ell^2(\mathbb{Z}_+^2) = \bigoplus_{n=0}^\infty \mathcal{K}(n).
\end{equation}
The orthogonal decomposition in (\ref{homog}), visualized in Figure 2(i), is called the homogeneous decomposition of $\ell^2(\mathbb{Z}_+^2)$. \ This decomposition will be very useful in the sequel, particularly when analyzing operators of the form $T_j^{\ast}T_i$ and $T_jT_i^{\ast}$, where $i,j=1,2$, which appear as entries in $L$ and $R$. \ In the context of $2$--variable weighted shifts, the homogeneous decomposition was previously used in \cite[Proposition 2.6]{CMX}.

\setlength{\unitlength}{.5mm} \psset{unit=.5mm} 
\begin{figure}[th]
\begin{center}
\begin{picture}(170,95)

\rput(-45,43){
\psline[linestyle=dotted](20,20)(85,20)
\psline[linestyle=dotted](20,40)(83,40)
\psline[linestyle=dotted](20,60)(83,60) 
\psline[linestyle=dotted](20,80)(83,80)


\psline[linestyle=dotted](20,40)(20,85) 
\psline[linestyle=dotted](20,20)(20,85)
\psline[linestyle=dotted](40,20)(40,83)
\psline[linestyle=dotted](60,20)(60,83) 
\psline[linestyle=dotted](80,20)(80,83)

\put(14,12){\footnotesize{$(0,0)$}}
\put(34.5,12){\footnotesize{$(1,0)$}}
\put(54.5,12){\footnotesize{$(2,0)$}}
\put(74.5,12){\footnotesize{$(3,0)$}}

\psline[linecolor=blue](40,20)(20,40)
\psline[linecolor=red](60,20)(20,60)
\psline[linecolor=green](80,20)(20,80)

\put(49,53){${\color{darkgreen}\mathcal{K}(3)}$}
\put(48,34){${\color{red}\mathcal{K}(2)}$}
\put(29,32){${\color{blue}\mathcal{K}(1)}$}
\put(21,23){${\color{black}\mathcal{K}(0)}$}

\put(81,21){\footnotesize{$\cdots$}}

\put(81,41){\footnotesize{$\cdots$}}

\put(81,61){\footnotesize{$\cdots$}}

\put(27,81){\footnotesize{$\cdots$}}
\put(47,81){\footnotesize{$\cdots$}}
\put(67,81){\footnotesize{$\cdots$}}


\put(5,39){\footnotesize{$(0,1)$}}
\put(5,59){\footnotesize{$(0,2)$}}
\put(5,79){\footnotesize{$(0,3)$}}

\put(21,81){\footnotesize{$\vdots$}}

\put(41,81){\footnotesize{$\vdots$}}

\put(61,81){\footnotesize{$\vdots$}}

\pscircle[fillstyle=solid,fillcolor=black](20,20){2}
\pscircle[fillstyle=solid,fillcolor=blue](20,40){2}
\pscircle[fillstyle=solid,fillcolor=red](20,60){2}
\pscircle[fillstyle=solid,fillcolor=blue](40,20){2}
\pscircle[fillstyle=solid,fillcolor=red](40,40){2}
\pscircle[fillstyle=solid,fillcolor=red](60,20){2}
\pscircle[fillstyle=solid,fillcolor=darkgreen](80,20){2}
\pscircle[fillstyle=solid,fillcolor=darkgreen](60,40){2}
\pscircle[fillstyle=solid,fillcolor=darkgreen](40,60){2}
\pscircle[fillstyle=solid,fillcolor=darkgreen](20,80){2}
}

\put(-70,-8){(i) Homogeneous decomposition of $\ell^2(\mathbb{Z}_+)$}. 
\put(100,-8){(ii) Block matrix representation of $L$ and $R$}.


\setlength{\unitlength}{.4mm} \psset{unit=.4mm} 

\rput(155,2){
\psline[linestyle=solid](10,10)(90,10)
\psline[linestyle=dotted](10,20)(90,20)
\psline[linestyle=dotted](10,30)(90,30)
\psline[linestyle=dotted](10,40)(90,40)
\psline[linestyle=dotted](10,50)(90,50)
\psline[linestyle=dotted](10,60)(90,60)
\psline[linestyle=dotted](10,70)(90,70)
\psline[linestyle=dotted](10,80)(90,80)
\psline[linecolor=red,linestyle=solid](10,80)(40,80)
\psline[linecolor=red,linestyle=solid](40,40)(80,40)
\psline[linestyle=solid](10,90)(90,90)
\psline[linecolor=red,linestyle=solid](20,60)(60,60)
\psline[linecolor=red,linestyle=solid](60,20)(90,20)

\psline[linestyle=solid](10,10)(10,90)
\psline[linestyle=dotted](20,10)(20,90)
\psline[linestyle=dotted](30,10)(30,90)
\psline[linestyle=dotted](40,10)(40,90)
\psline[linestyle=dotted](50,10)(50,90)
\psline[linestyle=dotted](60,10)(60,90)
\psline[linestyle=dotted](70,10)(70,90)
\psline[linestyle=dotted](80,10)(80,90)
\psline[linestyle=solid](90,10)(90,90)
\psline[linecolor=red,linestyle=solid](40,80)(40,40)
\psline[linecolor=red,linestyle=solid](60,60)(60,20)
\psline[linecolor=red,linestyle=solid](20,90)(20,60)
\psline[linecolor=red,linestyle=solid](80,40)(80,10)

}

\end{picture}
\caption{}
\end{center}
\end{figure}

In the special case of $2$--variable weighted shifts, the finite-dimensional spaces $\mathcal{K}(n)$ are actually reducing subspaces for $L$ and $R$. \ The restrictions of these operator matrices to $\mathcal{K}(n)$ are unitarily equivalent to a direct sum of $1 \times 1$ and $2 \times 2$ matrices, as established in the following result. \ In the special case of $n=3$, the restrictions of $L$ and $R$ to the reducing subspace $\mathcal{K}(3)$ are unitarily equivalent to $8 \times 8$ block diagonal matrices of the form $(1 \times 1) \oplus (2 \times 2) \oplus (2 \times 2) \oplus (2 \times 2) \oplus (1 \times 1)$, as shown in Figure 2(ii).
 
\begin{theorem}
\label{Thm6} Consider a $2$--variable weighted shift $W_{(\alpha ,\beta)}=(T_{1},T_{2})$. \ Then for $n\geq 0$,
\begin{equation}
\begin{tabular}{l}
$L|_{\mathcal{K}(n)}\cong (\alpha _{(0,n)}^{2}) \oplus \left[\bigoplus_{i=1}^n 
\left(
\begin{array}{cc}
\alpha_{(i,n-i)}^2 & \alpha_{(i-1,n-i+1)}\beta_{(i,n-i)} \\
\alpha_{(i-1,n-i+1)}\beta_{(i,n-i)} & \beta_{(i-1,n-i+1)}^2 
\end{array}
\right)\right] \oplus (\beta _{(n,0)}^{2})$ \\
and \\
$R|_{\mathcal{K}(n)}\cong (0) \oplus \left[\bigoplus_{i=1}^n 
\left(
\begin{array}{cc}
\alpha_{(i-1,n-i)}^2 & \alpha_{(i-1,n-i)}\beta_{(i-1,n-i)} \\
\alpha_{(i-1,n-i)}\beta_{(i-1,n-i)} & \beta_{(i-1,n-i)}^2 
\end{array}
\right) \right]\oplus (0)$,
\end{tabular}
\label{Decomposition}
\end{equation}
where $L|_{\mathcal{K}(n)}$ (resp. $R|_{\mathcal{K}(n)}$) is the restriction of $L$ (resp. $R$) to the reducing subspace $\mathcal{K}(n)$.
\end{theorem}

\begin{proof}
As a first step, we calculate the actions of the operator matrices $L$ and $R$ on the basis vectors of $\mathcal{K}(n)$. \ 

{\bf Case} $n=0$: \ The canonical basis vector for $\mathcal{K}(0)$ is $e_{(0,0)}$, and $T_1^*T_1e_{(0,0)}=\alpha_{(0,0)}^2e_{(0,0)}$, $T_2^*T_1e_{(0,0)}=0$, $T_1^*T_2e_{(0,0)}=0$, and $T_2^*T_2e_{(0,0)}=\beta_{(0,0)}^2e_{(0,0)}$. \ On the other hand, $T_1T_1^*e_{(0,0)}=T_2T_1^*e_{(0,0)}=T_1T_2^*e_{(0,0)} = T_2T_2^*e_{(0,0)}=0$. \ Thus,
$$
L|_{\mathcal{K}(0)} = \left(
\begin{array}{cc}
\alpha _{(0,0)}^{2} & 0 \\
0 & \beta_{(0,0)}^2
\end{array}%
\right) \quad \quad \quad \textrm{ and } 
R|_{\mathcal{K}(0)} = \left(
\begin{array}{cc}
0 & 0 \\
0 & 0
\end{array}
\right)
$$
{\bf Case} $n=1$: \ Here $e_{(0,1)}$ and $e_{(1,0)}$ are the canonical basis vectors for $\mathcal{K}(1)$. \ We obtain
$$
L|_{\mathcal{K}(1)} = \left(
\begin{array}{cccc}
\alpha _{(0,1)}^{2} & 0 & 0 & 0 \\
0 & \alpha_{(1,0)}^2 & \alpha_{(0,1)}\beta_{(1,0)} & 0 \\
0 & \alpha_{(0,1)}\beta_{(1,0)} & \beta_{(0,1)}^2 & 0 \\
0 & 0 & 0 & \beta_{(1,0)}^2 
\end{array}%
\right), \; 
R|_{\mathcal{K}(0)} = \left(
\begin{array}{cccc}
0 & 0 & 0 & 0 \\
0 & \alpha_{(0,0)}^2 & \alpha_{(0,0)}\beta_{(0,0)} & 0 \\
0 & \alpha_{(0,0)}\beta_{(0,0)} & \beta_{(0,0)}^2 & 0 \\
0 & 0 & 0 & 0 
\end{array}%
\right).
$$
{\bf Case} $n = 2$: \ $\mathcal{K}(2)$ is $3$--dimensional, with canonical basis vectors $e_{(0,2)}$, $e_{(1,1)}$ and $e_{(2,0)}$. \ We readily obtain
$$
L|_{\mathcal{K}(2)} = \left(
\begin{array}{cccccc}
\alpha _{(0,2)}^{2} & 0 & 0 & 0 & 0 & 0 \\
0 & \alpha_{(1,1)}^2 & 0 & \alpha_{(0,2)}\beta_{(1,1)} & 0 & 0 \\
0 & 0 & \alpha_{(2,0)}^2 & 0 & \alpha_{(1,1)}\beta_{(2,0)} & 0 \\
0 & \alpha_{(0,2)}\beta_{(1,1)} & 0 & \beta_{(0,2)}^2 & 0 & 0 \\
0 & 0 & \alpha_{(1,1)}\beta_{(2,0)} & 0 & \beta_{(1,1)}^2 & 0 \\
0 & 0 & 0 & 0 & 0 & \beta_{(2,0)}^2
\end{array}
\right) 
$$
and
$$ 
R|_{\mathcal{K}(2)} = \left(
\begin{array}{cccccc}
0 & 0 & 0 & 0 & 0 & 0 \\
0 & \alpha_{(0,1)}^2 & 0 & \alpha _{(0,1)} \beta _{(0,1)} & 0 & 0 \\
0 & 0 & \alpha_{(1,0)}^2 & 0 & \alpha_{(1,0)} \beta_{(1,0)} & 0 \\
0 & \alpha_{(0,1)} \beta_{(0,1)} & 0 & \beta_{(0,1)}^2 & 0 & 0 \\
0 & 0 & \alpha_{(1,02)}\beta_{(1,0)} & 0 & \beta_{(1,0)}^2 & 0 \\
0 & 0 & 0 & 0 & 0 & 0 \\
\end{array}%
\right). 
$$
If we now interchange the third and fourth rows and columns in both $6 \times 6$ matrices, we \linebreak obtain the modified matrices
$$
\widetilde{L|_{\mathcal{K}(2)}} = \left(
\begin{array}{cccccc}
\alpha _{(0,2)}^{2} & 0 & 0 & 0 & 0 & 0 \\
0 & \alpha_{(1,1)}^2 & \alpha_{(0,2)}\beta_{(1,1)} & 0 & 0 & 0 \\
0 & \alpha_{(0,2)}\beta_{(1,1)} & \beta_{(0,2)}^2 & 0 & 0 & 0 \\
0 & 0 & 0 & \alpha_{(2,0)}^2 & \alpha_{(1,1)}\beta_{(2,0)} & 0 \\
0 & 0 & 0 & \alpha_{(1,1)}\beta_{(2,0)} & \beta_{(1,1)}^2 & 0 \\
0 & 0 & 0 & 0 & 0 & \beta_{(2,0)}^2
\end{array}
\right)
$$
and
$$ 
\widetilde{R|_{\mathcal{K}(2)}} = \left(
\begin{array}{cccccc}
0 & 0 & 0 & 0 & 0 & 0 \\
0 & \alpha_{(0,1)}^2 & \alpha _{(0,1)} \beta _{(0,1)} & 0 & 0 & 0 \\
0 & \alpha_{(0,1)} \beta_{(0,1)} & \beta_{(0,1)}^2 & 0 & 0 & 0 \\
0 & 0 & 0 & \alpha_{(1,0)}^2 & \alpha_{(1,0)} \beta_{(1,0)} & 0 \\
0 & 0 & 0 & \alpha_{(1,02)}\beta_{(1,0)} & \beta_{(1,0)}^2 & 0 \\
0 & 0 & 0 & 0 & 0 & 0 \\
\end{array}%
\right). 
$$
As a result, we see that $L|_{\mathcal{K}(2)}$ and $R|_{\mathcal{K}(2)}$ are unitarily equivalent to the direct sums
$$
(\alpha _{(0,2)}^{2}) \oplus 
\left(
\begin{array}{cc}
\alpha_{(1,1)}^2 & \alpha_{(0,2)}\beta_{(1,1)} \\
\alpha_{(0,2)}\beta_{(1,1)} & \beta_{(0,2)}^2 
\end{array}
\right) \oplus
\left(
\begin{array}{cc}
\alpha_{(2,0)}^2 & \alpha_{(1,1)} \beta_{(2,0)} \\
\alpha_{(1,1)}\beta_{(2,0)} & \beta_{(1,1)}^2
\end{array}
\right) \oplus (\beta_{(2,0)}^2)
$$
and
$$
(0) \oplus 
\left(
\begin{array}{cc}
\alpha_{(0,1)}^2 & \alpha_{(0,1)}\beta_{(0,1)} \\
\alpha_{(0,1)}\beta_{(0,1)} & \beta_{(0,1)}^2 
\end{array}
\right) \oplus
\left(
\begin{array}{cc}
\alpha_{(1,0)}^2 & \alpha_{(1,0)} \beta_{(1,0)} \\
\alpha_{(1,0)}\beta_{(1,0)} & \beta_{(1,0)}^2
\end{array}
\right) \oplus (0),
$$
respectively.

{\bf Case} $n \ge 3$: \ In a completely similar way, we can establish that, up to interchange of rows and columns $3$ and $4$, $5$ and $6$, and so on, we have
\begin{eqnarray*}
L|_{\mathcal{K}(n)} \!\!\!\! &\cong& 
\!\!\!\! (\alpha _{(0,n)}^{2}) \oplus 
\left(
\begin{array}{cc}
\alpha_{(1,n-1)}^2 & \alpha_{(0,n)}\beta_{(1,n-1)} \\
\alpha_{(0,n)}\beta_{(1,n-1)} & \beta_{(0,n)}^2 
\end{array}
\right) \oplus
\left(
\begin{array}{cc}
\alpha_{(2,n-2)}^2 & \alpha_{(1,n-1)} \beta_{(2,n-2)} \\
\alpha_{(1,n-1)}\beta_{(2,n-2)} & \beta_{(1,n-1)}^2
\end{array}
\right) \\
&&
\!\!\!\!\!\!\!\!\! \oplus \cdots \oplus \!\!
\left(
\begin{array}{cc}
\alpha_{(n-1,1)}^2 & \alpha_{(n-2,2)}\beta_{(n-1,1)} \\
\alpha_{(n-2,2)}\beta_{(n-1,1)} & \beta_{(n-2,2)}^2 
\end{array}
\right) \!\! \oplus \!\!
\left(
\begin{array}{cc}
\alpha_{(n,0)}^2 & \alpha_{(n-1,1)}\beta_{(n,0)} \\
\alpha_{(n-1,1)}\beta_{(n,0)} & \beta_{(n-1,1)}^2 
\end{array}
\right) \!\! \oplus \!\!(\beta_{(n,0)}^2) \\
&=& (\alpha _{(0,n)}^{2}) \oplus \left[\bigoplus_{i=1}^n 
\left(
\begin{array}{cc}
\alpha_{(i,n-i)}^2 & \alpha_{(i-1,n-i+1)}\beta_{(i,n-i)} \\
\alpha_{(i-1,n-i+1)}\beta_{(i,n-i)} & \beta_{(i-1,n-i+1)}^2 
\end{array}
\right)\right] \oplus (\beta _{(n,0)}^{2}).
\end{eqnarray*}

Similarly, 
\begin{eqnarray*}
R|_{\mathcal{K}(n)} \!\!\!\! &\cong& 
\!\!\!\! (0) \oplus \!\!
\left(
\begin{array}{cc}
\alpha_{(0,n-1)}^2 & \alpha_{(0,n-1)}\beta_{(0,n-1)} \\
\alpha_{(0,n-1)}\beta_{(0,n-1)} & \beta_{(0,n-1)}^2 
\end{array}
\right) \!\! \oplus \!\!
\left(
\begin{array}{cc}
\alpha_{(1,n-2)}^2 & \alpha_{(1,n-2)} \beta_{(1,n-2)} \\
\alpha_{(1,n-2)}\beta_{(1,n-2)} & \beta_{(1,n-2)}^2
\end{array}
\right) \!\! \oplus \cdots \\
&&
\!\!\!\!\!\oplus
\left(
\begin{array}{cc}
\alpha_{(n-2,1)}^2 & \alpha_{(n-2,1)}\beta_{(n-2,1)} \\
\alpha_{(n-2,1)}\beta_{(n-2,1)} & \beta_{(n-2,1)}^2 
\end{array}
\right) \oplus
\left(
\begin{array}{cc}
\alpha_{(n-1,0)}^2 & \alpha_{(n-1,0)}\beta_{(n-1,0)} \\
\alpha_{(n-1,0)}\beta_{(n-1,0)} & \beta_{(n-1,0)}^2 
\end{array}
\right)\oplus (0) \\
&=& (0) \oplus \left[\bigoplus_{i=1}^n 
\left(
\begin{array}{cc}
\alpha_{(i-1,n-i)}^2 & \alpha_{(i-1,n-i)}\beta_{(i-1,n-i)} \\
\alpha_{(i-1,n-i)}\beta_{(i-1,n-i)} & \beta_{(i-1,n-i)}^2 
\end{array}
\right) \right]\oplus (0).
\end{eqnarray*}
We have thus obtained (\ref{Decomposition}), which completes the proof of the theorem.
\end{proof}

\subsection{Positivity of $L$}

Recall that 
$$
L=\left(
\begin{array}{ll}
T_{1}^{\ast }T_{1} & T_{2}^{\ast }T_{1} \\
T_{1}^{\ast }T_{2} & T_{2}^{\ast }T_{2}
\end{array}
\right). \label{matrixL}
$$
In the special case when $T_2 = T_1^2$ (this happens when we are testing the $2$--hyponormality of $T_1$), one can rewrite 
$$
L = 
\left(
\begin{array}{ll}
T_{1}^{\ast }T_{1} & T_{1}^{\ast 2} T_{1} \\
T_{1}^{\ast }T_{1}^2 & T_{1}^{\ast 2} T_{1}^2
\end{array}
\right) = 
\left(
\begin{array}{ll}
T_{1}^{\ast } & 0 \\
0 & T_{1}^{\ast }
\end{array}
\right) 
\left(
\begin{array}{ll}
I & T_1^{\ast} \\
T_{1} & T_1^{\ast}T_1
\end{array}
\right) 
\left(
\begin{array}{ll}
T_{1} & 0 \\
0 & T_{1}
\end{array}
\right).  
$$
It follows that, in this case, the hyponormality of $T_1$ readily implies the positivity of $L$. \ The same is not true for arbitrary commuting pairs, even for $2$--variable weighted shifts, as we will see shortly. \ First, notice that, by Smul'jan's Test, the positivity of $L$ is determined by the existence of a contraction $E$ such that
\begin{equation} \label{smul}
T_2^{\ast}T_1 = \sqrt{T_1^{\ast}T_1} E \sqrt{T_2^{\ast}T_2}.
\end{equation}
While it is generally quite hard to obtain a concrete formula for the contraction $E$ from (\ref{smul}), for $2$--variable weighted shifts it is indeed possible to find such a formula. \ First, we can directly calculate the action of $T_2^{\ast}T_1$ on a canonical orthonormal basic vector $e_{\mathbf{k}} \equiv e_{(k_1,k_2)}$:

\begin{equation*}
\begin{tabular}{l}
$T_{2}^{\ast }T_{1}e_{(k_{1},k_{2})}=\sqrt{T_{1}^{\ast }T_{1}} E \sqrt{T_{2}^{\ast }T_{2}}
e_{(k_{1},k_{2})}=\sqrt{T_{1}^{\ast }T_{1}} E\beta
_{(k_{1},k_{2})}e_{(k_{1},k_{2})}$ \vspace{6pt} \\
$=\beta _{(k_{1},k_{2})} \sqrt{T_{1}^{\ast }T_{1}} E e_{(k_{1},k_{2})}.$ \vspace{6pt}
\end{tabular}
\end{equation*}
Since $T_{2}^{\ast }T_{1}e_{(k_{1},k_{2})} = 0$ whenever $k_2=0$, we restrict attention to the case $k_2 \ge 1$. \ Then 
\begin{align*}
T_{2}^{\ast }T_{1}e_{(k_{1},k_{2})} &= \alpha_{(k_{1},k_{2})}T_{2}^{\ast }e_{(k_{1}+1,k_{2})} \\
&= \alpha_{(k_{1},k_{2})}\beta _{(k_{1}+1,k_{2}-1)}e_{(k_{1}+1,k_{2}-1)} \vspace{4pt}
\end{align*}
and it follows that $E$ maps $e_{(k_{1},k_{2})}$ to a multiple of $e_{(k_{1}+1,k_{2}-1)}$,
that is, 
$$
Ee_{(k_{1},k_{2})}=\lambda _{(k_{1},k_{2})}e_{(k_{1}+1,k_{2}-1)}.
$$
Therefore,
\begin{equation*}
\begin{tabular}{l}
$\alpha _{(k_{1},k_{2})}\beta
_{(k_{1}+1,k_{2}-1)}e_{(k_{1}+1,k_{2}-1)}=\alpha _{(k_{1}+1,k_{2}-1)}\beta
_{(k_{1},k_{2})}\lambda _{(k_{1},k_{2})}e_{(k_{1}+1,k_{2}-1)}$ 
\end{tabular}
\end{equation*}
and, as a result, 
$$
\lambda _{(k_{1},k_{2})}=\frac{\alpha _{(k_{1},k_{2})}\beta
_{(k_{1}+1,k_{2}-1)}}{\beta _{(k_{1},k_{2})}\alpha _{(k_{1}+1,k_{2}-1)}}.
$$
Thus, the necessary and sufficient condition for $E$ to be a
contraction is
\begin{equation}
\begin{tabular}{l}
$\alpha _{(k_{1},k_{2})}\beta _{(k_{1}+1,k_{2}-1)}\leq \beta
_{(k_{1},k_{2})}\alpha _{(k_{1}+1,k_{2}-1)}$ 
\end{tabular}
\label{con11}
\end{equation}
for all $(k_{1},k_{2})$ with $k_{2}\geq 1$; or equivalently,
\begin{equation}
\begin{tabular}{l}
$\cbl{\alpha _{(m_{1},m_{2}+1)}\beta _{(m_{1}+1,m_{2})}}\leq \cre{\beta
_{(m_{1},m_{2}+1)}\alpha _{(m_{1}+1,m_{2})}}$ 
\end{tabular}
\label{con12}
\end{equation}
for all $(m_{1},m_{2}) \in \ell^2(\mathbb{Z}_+^2)$. \ We visualize this sequence of inequalities in Figure \ref{Figure 2}.

\setlength{\unitlength}{1mm} \psset{unit=1mm}
\begin{figure}[th]
\begin{center}
\begin{picture}(135,51)

\rput(0,15){
\psline(30,20)(50,20)
\psline(10,40)(30,40)
\psline(10,40)(10,60)
\psline(30,20)(30,40)

\put(30,20){\pscircle*(0,0){1}}
\put(50,20){\pscircle*(0,0){1}} \put(10,40){\pscircle*(0,0){1}}
\put(30,40){\pscircle*(0,0){1}} \put(10,60){\pscircle*(0,0){1}}

\put(22,16){\footnotesize{$\mathbf{k}+\mathbf{\varepsilon}_{1}-\mathbf{\varepsilon}_{2}$}}
\put(43,16){\footnotesize{$\mathbf{k}+2\mathbf{\varepsilon}_{1}-\mathbf{\varepsilon}_{2}$}}

\put(34,22){\footnotesize{$\alpha_{\mathbf{k}+\mathbf{\varepsilon}_{1}-\mathbf{\varepsilon}_{2}}$}}

\put(18,42){\footnotesize{$\alpha_{\mathbf{k}}$}}

\put(4,39){\footnotesize{$\mathbf{k}$}}
\put(-1,59){\footnotesize{$\mathbf{k}+\mathbf{\varepsilon}_{2}$}}
\put(34,39){\footnotesize{$\mathbf{k}+\mathbf{\varepsilon}_{1}$}}

\put(10,49){\footnotesize{$\beta_{\mathbf{k}}$}}

\put(30,29){\footnotesize{$\beta_{\mathbf{k}+\mathbf{\varepsilon}_{1}-\mathbf{\varepsilon}_{2}}$}}



\rput(60,30){

\psline(30,20)(50,20)
\psline(10,40)(30,40)
\psline(10,40)(10,60)
\psline(30,20)(30,40)

\put(30,20){\pscircle*(0,0){1}}
\put(50,20){\pscircle*(0,0){1}} \put(10,40){\pscircle*(0,0){1}}
\put(30,40){\pscircle*(0,0){1}} \put(10,60){\pscircle*(0,0){1}}

\put(24,16){\footnotesize{$\mathbf{m}+\mathbf{\varepsilon}_{1}$}}
\put(41,16){\footnotesize{$\mathbf{m}+2\cdot\mathbf{\varepsilon}_{1}$}}

\put(36,22){\footnotesize{\cre{$\alpha_{\mathbf{m}+\mathbf{\varepsilon}_{1}}$}}}

\put(16,42){\footnotesize{\cbl{$\alpha_{\mathbf{m}+\mathbf{\varepsilon}_{2}}$}}}

\put(-3,39){\footnotesize{$\mathbf{m}+\mathbf{\varepsilon}_{2}$}}
\put(-6,59){\footnotesize{$\mathbf{m}+2 \cdot \mathbf{\varepsilon}_{2}$}}
\put(34,39){\footnotesize{$\mathbf{m}+\mathbf{\varepsilon}_{1}+\mathbf{\varepsilon}_{2}$}}

\put(10,49){\footnotesize{\cre{$\beta_{\mathbf{m}+\mathbf{\varepsilon}_{2}}$}}}

\put(30,29){\footnotesize{\cbl{$\beta_{\mathbf{m}+\mathbf{\varepsilon}_{1}}$}}}

\rput(3,35){
\put(50,60){\footnotesize{$\mathbf{\varepsilon}_1 = (1,0)$}}
\put(50,55){\footnotesize{$\mathbf{\varepsilon}_2 = (0,1)$}}
\put(50,50){\footnotesize{{\color{brown}{$\mathbf{m}:=\mathbf{k}-\mathbf{\varepsilon_2}$}}}}
\put(50,45){\footnotesize{{\textrm{(change of variable)}}}}

\psline(48,64)(80,64)
\psline(48,42.5)(80,42.5)
\psline(48,42.5)(48,64)
\psline(80,42.5)(80,64)
}
}
}
\end{picture}
\end{center}
\caption{Positivity of $L$.}
\label{Figure 2}
\end{figure}
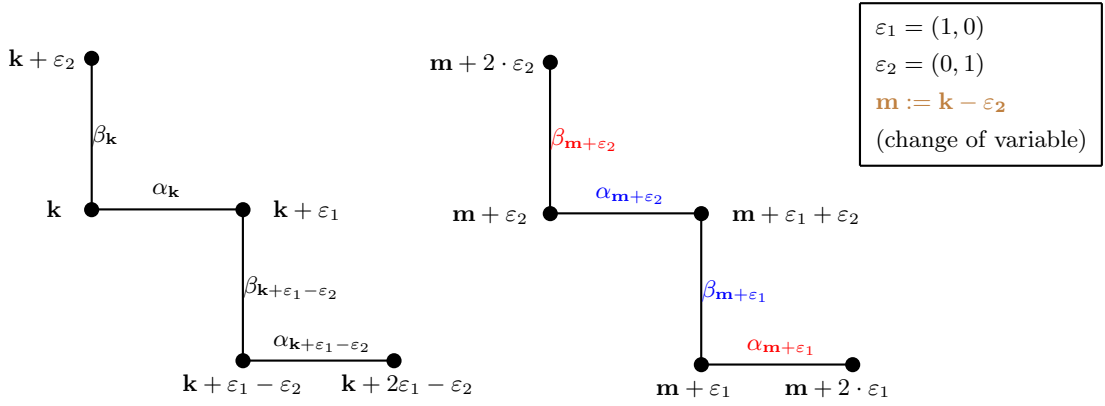

\noindent At the level of moments, we have
\begin{equation*}
\begin{tabular}{l}
$\alpha _{(k_{1},k_{2})}\beta _{(k_{1}+1,k_{2}-1)}\leq \beta
_{(k_{1},k_{2})}\alpha _{(k_{1}+1,k_{2}-1)}$ \\
$\Longleftrightarrow \frac{\gamma _{(k_{1}+1,k_{2})}}{\gamma _{(k_{1},k_{2})}}%
\cdot \frac{\gamma _{(k_{1}+1,k_{2})}}{\gamma _{(k_{1}+1,k_{2}-1)}}\leq
\frac{\gamma _{(k_{1},k_{2}+1)}}{\gamma _{(k_{1},k_{2})}}\cdot \frac{\gamma
_{(k_{1}+2,k_{2}-1)}}{\gamma _{(k_{1}+1,k_{2}-1)}}$ \\
$\Longleftrightarrow \gamma _{(k_{1}+1,k_{2})}^{2}\leq \gamma
_{(k_{1},k_{2}+1)}\cdot \gamma _{(k_{1}+2,k_{2}-1)}$ \\
$\Longleftrightarrow \gamma _{(k_{1}+1,k_{2}+1)}^{2}\leq \gamma
_{(k_{1},k_{2}+2)}\cdot \gamma _{(k_{1}+2,k_{2})}$ \quad \quad ($k_{1},k_{2} \ge 0$).
\end{tabular}
\end{equation*}
We thus have, for all $\mathbf{k}\in \mathbb{Z}_{+}^{2}$, 
\begin{equation}
L\geq 0\Longleftrightarrow \gamma _{\mathbf{k}+\varepsilon_1+\varepsilon_2}^{2}\leq \gamma
_{\mathbf{k}+2\varepsilon_2}\cdot \gamma _{\mathbf{k}+2\varepsilon_1}.  \label{con3}
\end{equation}

Summarizing the preceding observations, we are ready to state the following result. 

\begin{proposition}
Let $W_{(\alpha,\beta)}$ be a $2$--variable weighted shift, and let $L$ be the associated $2 \times 2$ operator matrix in (\ref{matrixL}). \ Then 
$$
L \ge 0 \Longleftrightarrow \alpha_{\mathbf{k}+\mathbf{\varepsilon}_2} \cdot \beta_{\mathbf{k}+\mathbf{\varepsilon}_1} \le \alpha_{\mathbf{k}+\mathbf{\varepsilon}_1} \cdot \beta_{\mathbf{k}+\mathbf{\varepsilon}_2} \Longleftrightarrow
\gamma_{\mathbf{k}+\mathbf{\varepsilon}_1+\mathbf{\varepsilon}_2}^2 \le \gamma_{\mathbf{k}+2\mathbf{\varepsilon}_1} \cdot \gamma_{\mathbf{k}+2\mathbf{\varepsilon}_2}, 
$$
for all $\mathbf{k} \in \mathbb{Z}_+^2$.
\end{proposition}

Now observe that, by Lemma \ref{khypo}, 
$$
W_{(\alpha ,\beta )} \textrm{ is hyponormal } \Longleftrightarrow \left(
\begin{array}{ccc}
\gamma_{\mathbf{k}} & \gamma _{\mathbf{k+\varepsilon_1}} & \gamma
_{\mathbf{k+\varepsilon_2}} \\
\gamma _{\mathbf{k+\varepsilon_1}} & \gamma _{\mathbf{k}+2\varepsilon_1} & \gamma
_{\mathbf{k+\varepsilon_1+\varepsilon_2}} \\
\gamma _{\mathbf{k+\varepsilon_2}} & \gamma _{\mathbf{k+\varepsilon_1+\varepsilon_2}} & \gamma
_{\mathbf{k}+2\varepsilon_2}
\end{array}
\right) \geq 0.
$$
Focusing on the lower right $2 \times 2$ minor, which must be nonnegative, yields the inequality in (\ref{con3}). \ It follows that the hyponormality of $W_{(\alpha,\beta)}$ readily implies (\ref{con3}), as expected. \ On the other hand, we will now exhibit a nontrivial, but rather simple, $2$--variable weighted shift $W_{(\alpha,\beta)} \equiv (T_1,T_2)$ such that $T_i$ is hyponormal ($i=1,2$), but $L$ is not positive (cf. Example \ref{Ex1}).

In $\ell^2(\mathbb{Z}_+^2)$, we let $\mathcal{M}_{i}$ $\left( \text{resp. }\mathcal{N}_{j}\right ) $ be the subspace of $\ell ^{2}(\mathbb{Z}_{+}^{2})$ spanned by the canonical orthonormal basis associated to indices $\mathbf{k}=(k_{1},k_{2})$ with $k_{1}\geq 0$ and $k_{2}\geq i$ (resp. $k_{1}\geq j$ and $k_{2}\geq 0$). \ We will write $\mathcal{M}_{1}$ simply as $\mathcal{M}$ and $\mathcal{N}_{1}$ as $\mathcal{N}$. \ Given a $2$--variable weighted shift $W_{\left( \alpha ,\beta \right) }$, we let $W_{\left( \alpha ,\beta \right) }|_{\mathcal{M}\cap \mathcal{N}}$ denote the restriction of $W_{(\alpha ,\beta )}$ to the invariant subspace $\mathcal{M}\cap \mathcal{N}$. \ Now we consider:

\begin{example} \label{ex215}
\label{Ex1} (i) \ Let $0<a,b<1$ with $a<b$. \ Consider a $2$--variable weighted shift $W_{\left( \alpha ,\beta \right) }\equiv(T_{1},T_{2})$ whose diagram is given in Figure 4(i); that is, 
\newline (i) \ $W_{\left( \alpha ,\beta \right)}|_{\mathcal{M}_{2}\cap \mathcal{N}_{2}}$ is the Helton-Howe shift;
\newline (ii) \ $W_{\left( \alpha ,\beta \right) }|_{\mathcal{M}_{2}}\cong (I\otimes \mathrm{shift}(b^2,b^2,1,1,\ldots ),U_{+}\otimes I)$;  
\newline (iii) \ $W_{\left( \alpha ,\beta\right) }|_{\mathcal{N}_{2}}\cong (I\otimes U_{+},\mathrm{shift}(b^2,b^2,1,1,\ldots )\otimes I)$; \ 
\newline (iv) \ $\alpha_{(0,0)}:=a^2$, $\alpha_{(1,0)}:=a^2$, $\alpha_{(0,1)}:=ab$, $\alpha_{(1,1)}:=ab$; 
\newline (v) \ $\beta_{(0,0)}:=a^2$, $\beta_{(0,1)}:=a^2$, $\beta_{(1,0)}:=ab$, $\beta_{(1,1)}:=ab$; and
\newline (vi) \ the remaining weights determined by the commutativity of $T_1$ and $T_2$.

\noindent It is straightforward to verify that $T_{1}$ and $T_{2}$ are both hyponormal. \ We will now show that $L$ is not positive. \ To do this, we use (\ref{con11}), as follows: at the lattice point $(k_1,k_2)=(0,1)$, we have $\alpha_{(0,1)}\beta_{(1,0)}=a^2b^2$ and $\beta_{(0,1)}\alpha_{(1,0)}=a^4$; since $a<b$, we see that $a^2b^2 > a^4$, so (\ref{con11}) does not hold. \ (At all other lattice points, the inequality in (\ref{con11}) does hold.) \qed
\end{example}

In the next example, we show that, even for $2$--variable weighted shifts, the semi-hyponormality of a commuting pair $(T_1,T_2)$ does not guarantee the hyponormality of both components $T_1$ and $T_2$. \ First, recall that the positivity of $L$ and $R$ can be detected by their restrictions to the reducing subspaces $\mathcal{K}(n) \; (n \ge 0)$, as shown in Theorem \ref{Thm6}.

\begin{example} \label{ex216}
Consider the $2$-variable weighted shift $W_{(\alpha,\beta)}$ whose weight diagram is shown in Figure 4(ii), where the parameters $a$ and $b$ are positive numbers, and $a > \dfrac{1}{2}$. 

\setlength{\unitlength}{1mm} \psset{unit=1mm}
\begin{figure}[th]
\begin{center}
\begin{picture}(135,70)

\rput(-70,35){
\psline{->}(90,14)(105,14)
\put(95,10){$\rm{T}_1$}
\psline{->}(66,35)(66,50)
\put(61,42){$\rm{T}_2$}

\psline{->}(75,20)(130,20)
\psline(75,35)(127,35)
\psline(75,50)(127,50)
\psline(75,65)(127,65)

\psline{->}(75,20)(75,72)
\psline(90,20)(90,71)
\psline(105,20)(105,71)
\psline(120,20)(120,71)

\put(71,16){\footnotesize{$(0,0)$}}
\put(67,34){\footnotesize{$(0,1)$}}
\put(67,49){\footnotesize{$(0,2)$}}
\put(67,64){\footnotesize{$(0,3)$}}

\put(86,16){\footnotesize{$(1,0)$}}
\put(101,16){\footnotesize{$(2,0)$}}
\put(116,16){\footnotesize{$(3,0)$}}

\put(82,21){\footnotesize{$a^2$}}
\put(97,21){\footnotesize{$a^2$}}
\put(112,21){\footnotesize{$1$}}
\put(125,21){\footnotesize{$1$}}

\put(82,36){\footnotesize{$ab$}}
\put(97,36){\footnotesize{$ab$}}
\put(112,36){\footnotesize{$1$}}
\put(125,36){\footnotesize{$1$}}

\put(82,51){\footnotesize{$b^2$}}
\put(97,51){\footnotesize{$b^2$}}
\put(112,51){\footnotesize{$1$}}
\put(125,51){\footnotesize{$1$}}

\put(82,66){\footnotesize{$b^2$}}
\put(97,66){\footnotesize{$b^2$}}
\put(112,66){\footnotesize{$1$}}
\put(125,66){\footnotesize{$1$}}

\put(76,26){\footnotesize{$a^2$}}
\put(76,41){\footnotesize{$a^2$}}
\put(76,56){\footnotesize{$1$}}
\put(76,69){\footnotesize{$1$}}

\put(91,26){\footnotesize{$ab$}}
\put(91,41){\footnotesize{$ab$}}
\put(91,56){\footnotesize{$1$}}
\put(91,69){\footnotesize{$1$}}

\put(106,26){\footnotesize{$b^2$}}
\put(106,41){\footnotesize{$b^2$}}
\put(106,56){\footnotesize{$1$}}
\put(106,69){\footnotesize{$1$}}

\put(121,26){\footnotesize{$b^2$}}
\put(121,41){\footnotesize{$b^2$}}
\put(121,56){\footnotesize{$1$}}
\put(121,69){\footnotesize{$1$}}

}


%
\put(0,5){(i)}


\put(80,5){(ii)}

\rput(4,35){
\psline{->}(95,14)(110,14)
\put(102,10){$\rm{T}_1$}
\psline{->}(70,35)(70,50)
\put(65,42){$\rm{T}_2$}

\psline{->}(80,20)(130,20)
\psline(80,35)(128,35)
\psline(80,50)(128,50)
\psline(80,65)(128,65)

\psline{->}(80,20)(80,70)
\psline(95,20)(95,68)
\psline(110,20)(110,68)
\psline(125,20)(125,68)

\put(75,16){\footnotesize{$(0,0)$}}
\put(91,16){\footnotesize{$(1,0)$}}
\put(106,16){\footnotesize{$(2,0)$}}
\put(121,16){\footnotesize{$(3,0)$}}

\put(72,34){\footnotesize{$(0,1)$}}
\put(72,49){\footnotesize{$(0,2)$}}
\put(72,64){\footnotesize{$(0,3)$}}

\put(85,21){\footnotesize{$a$}}
\put(100,21){\footnotesize{$1$}}
\put(115,21){\footnotesize{$1$}}
\put(126,21){\footnotesize{$\cdots$}}

\put(85,36){\footnotesize{$1$}}
\put(100,36){\footnotesize{$1$}}
\put(115,36){\footnotesize{$1$}}
\put(126,36){\footnotesize{$\cdots$}}

\put(85,51){\footnotesize{$1$}}
\put(100,51){\footnotesize{$1$}}
\put(115,51){\footnotesize{$1$}}
\put(126,51){\footnotesize{$\cdots$}}

\put(85,66){\footnotesize{$\cdots$}}
\put(100,66){\footnotesize{$\cdots$}}
\put(115,66){\footnotesize{$\cdots$}}
\put(126,66){\footnotesize{$\cdots$}}

\put(81,26){\footnotesize{$b$}}
\put(81,41){\footnotesize{$2b$}}
\put(81,56){\footnotesize{$2b$}}
\put(81,66){\footnotesize{$\vdots$}}

\put(96,26){\footnotesize{$\dfrac{b}{a}$}}
\put(96,41){\footnotesize{$2b$}}
\put(96,56){\footnotesize{$2b$}}
\put(96,66){\footnotesize{$\vdots$}}

\put(111,26){\footnotesize{$\dfrac{b}{a}$}}
\put(111,41){\footnotesize{$2b$}}
\put(111,56){\footnotesize{$2b$}}
\put(111,66){\footnotesize{$\vdots$}}

}



\end{picture}
\caption{(i) Weight diagram for Example \ref{ex215}; (ii) Weight diagram for Example \ref{ex216}.}
\end{center} \label{Fig 10}
\end{figure}
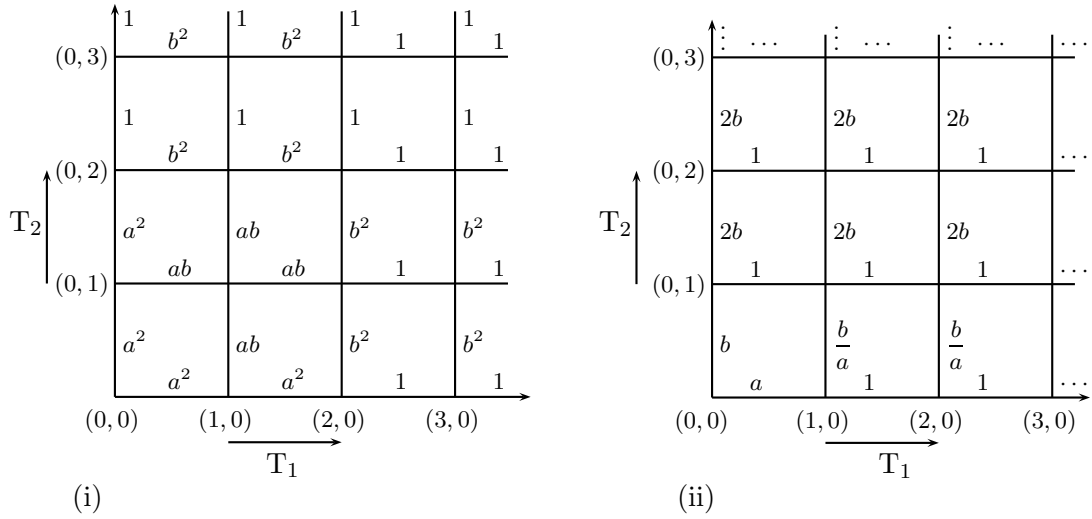

Observe first that $T_2$ is subnormal, and that 
$$
W_{\left(\alpha ,\beta \right) }|_{\mathcal{M}}\cong (I\otimes U_{+},\,2b \cdot U_+\otimes I)
$$
and
$$
W_{\left( \alpha ,\beta \right) }|_{\mathcal{N}}\cong (I\otimes U_{+},\, \shift(\dfrac{b}{a},2b,2b,\ldots) \otimes I)
$$
are both subnormal (and hence hyponormal). \ Thus, the inequalities $L|_{\mathcal{K}(n)} \ge R|_{\mathcal{K}(n)}$ for $n \ge 2$, are guaranteed. \ Therefore, to determine the semi-hyponormality of $W_{(\alpha,\beta)}$, it suffices to verify that $L|_{\mathcal{K}(n)}\geq 0$ and $\sqrt{L|_{\mathcal{K}(n)}} \geq \sqrt{R|_{\mathcal{K}(n)}}$, for $n=0,1$. \ The case $n=0$ is straightforward, so we focus on $n=1$. \ From (\ref{Decomposition}) we obtain
$$
L|_{\mathcal{K}(1)}\geq 0 \Longleftrightarrow \left(
\begin{array}{cc}
1 & \frac{b}{a} \\
\frac{b}{a} & 4b^{2}%
\end{array}%
\right) \geq 0
$$
and 
$$
R|_{\mathcal{K}(1)}\Longleftrightarrow \left(
\begin{array}{cc}
a^{2} & ab \\
ab & b^2
\end{array}%
\right) \geq 0.
$$
\medskip
By (\ref{squareroot2}), to compute the square root of $L|_{\mathcal{K}(1)}$ we first need to calculate its determinant and trace; that is,  
$$
\det (L|_{\mathcal{K}(1)}) = 4b^2-\dfrac{b^2}{a^2} = \dfrac{b^2}{a^2} \cdot (4a^2-1)>0, \quad \textrm{ and } \tr (L|_{\mathcal{K}(1)}) = 1+4b^2>0.
$$
As we know from (\ref{squareroot2}), 
$$
\sqrt{L|_{\mathcal{K}(1)}} = \dfrac{1}{\sqrt{\tr (L|_{\mathcal{K}(1)})+2 \sqrt{\det (L|_{\mathcal{K}(1)})}}} \cdot (L|_{\mathcal{K}(1)} + \sqrt{\det (L|_{\mathcal{K}(1)}} \cdot I),
$$
that is,
$$ 
\sqrt{L|_{\mathcal{K}(1)}}=\dfrac{1}{\sqrt{1+4b^2+2 \cdot \frac{b}{a} \cdot \sqrt{4a^2-1}}} \cdot \left(
\begin{array}{cc}
1 + \frac{b}{a} \cdot \sqrt{4a^2-1} & \frac{b}{a} \\
& \\
\frac{b}{a} & 4b^2+\frac{b}{a} \cdot \sqrt{4a^2-1}
\end{array}
\right).
$$
Let $p:=\sqrt{4a^2-1}$ and $q:=\sqrt{1+4b^2+2 \cdot \frac{b}{a} \cdot p}$, and observe that both $p$ and $q$ are positive numbers, in view of the assumption $a > \frac{1}{2}$. \ It follows that  
\begin{equation*}
\begin{tabular}{l}
$\sqrt{L|_{\mathcal{K}(1)}} = \left(
\begin{array}{cc}
1 & \frac{b}{a} \\
& \\
\frac{b}{a} & 4b^{2}
\end{array}%
\right) ^{\frac{1}{2}}=\frac{1}{q}\left(
\begin{array}{cc}
1+\frac{b}{a} \cdot p & \frac{b}{a} \\
& \\
\frac{b}{a} & 4b^{2}+\frac{b}{a} \cdot p
\end{array}
\right)$.
\end{tabular}
\end{equation*}
Similarly, using (\ref{squareroot2}) we obtain
\begin{equation*}
\begin{tabular}{l}
$\sqrt{R|_{\mathcal{K}(1)}}=\left(
\begin{array}{cc}
a^{2} & ab \\
ab & b^2
\end{array}
\right) ^{\frac{1}{2}}=\frac{1}{\sqrt{a^2+b^2}}\left(
\begin{array}{cc}
a^2 & ab \\
ab & b^2
\end{array}
\right) $.
\end{tabular}
\end{equation*}
As we mentioned before, the semihyponormality of $W_{(\alpha,\beta)}$ is entirely determined by the inequality $\operatorname{SH}(a,b):=\sqrt{L|_{\mathcal{K}(1)}} - \sqrt{R|_{\mathcal{K}(1)}} \ge 0$. \ Now, 
\begin{equation*}
\operatorname{SH}(a,b)=\left(
\begin{array}{cc}
\frac{1}{q} \cdot (1+\frac{b}{a} \cdot p) - \frac{a^2}{\sqrt{a^2+b^2}} & \frac{1}{q} \cdot \frac{b}{a} - \frac{ab}{\sqrt{a^2+b^2}} \\
& \\
\frac{1}{q} \cdot \frac{b}{a} - \frac{ab}{\sqrt{a^2+b^2}} & \frac{1}{q} \cdot (4b^2+\frac{b}{a} \cdot p) - \frac{b^2}{\sqrt{a^2+b^2}}
\end{array}
\right).
\end{equation*}
The matrix-valued function $\operatorname{SH}:(0,+\infty)^2\longrightarrow M_2(\mathbb{C})$ (where $M_2(\mathbb{C})$ denotes the algebra of $2 \times 2$ complex matrices over the complex numbers), is clearly continuous in a neighborhood of $(1,1)$; moreover, at $(1,1)$, we have $p=\sqrt{3}$ and $q=\sqrt{5+2\sqrt{3}}$. \ Therefore, 
$$
\operatorname{SH}(1,1)=\left(
\begin{array}{cc}
\frac{1+p}{q} - \frac{1}{\sqrt{2}} & \frac{1}{q} - \frac{1}{\sqrt{2}} \\
& \\
\frac{1}{q} - \frac{1}{\sqrt{2}} & \frac{4+p}{q} - \frac{1}{\sqrt{2}}
\end{array}
\right) = \left(
\begin{array}{cc}
\frac{1+\sqrt{3}}{\sqrt{5+2\sqrt{3}}} - \frac{1}{\sqrt{2}} & \frac{1}{\sqrt{5+2\sqrt{3}}} - \frac{1}{\sqrt{2}} \\
& \\
\frac{1}{\sqrt{5+2\sqrt{3}}} - \frac{1}{\sqrt{2}} & \frac{4+\sqrt{3}}{\sqrt{5+2\sqrt{3}}} - \frac{1}{\sqrt{2}}
\end{array} \right).
$$
A simple calculation using {\it Mathematica} \cite{Wol} shows that both $\tr(\operatorname{SH}(1,1))$ and $\det(\operatorname{SH}(1,1)$ are positive at $(1,1)$. \ This implies that $\operatorname{SH}(1,1)$ is a positive and invertible matrix. \ By continuity, the same must be true in an open neighborhood $\mathcal{G}$ of $(1,1)$, so $W_{(\alpha,\beta)}$ remains semi-hyponormal for all values of $a$ and $b$ such that $(a,b) \in \mathcal{G}$. \ However, as soon as $a$ is bigger than $1$, the hyponormality of $T_1$ is lost, and therefore $W_{(\alpha,\beta)}$ cannot be hyponormal. \ In short, it is possible to secure values for $a$ and $b$ such that $W_{(\alpha,\beta)}$ is semi-hyponormal but not hyponormal; for a concrete example, take $a=b=\frac{21}{20}$; in this case, $\tr (\operatorname{SH}(1,1)) \cong 1.77469$ and $\det (\operatorname{SH}(1,1)) \cong 0.44924$. \ Therefore, $\operatorname{SH}(1,1) \ge 0$, but $T_1$ is not hyponormal.
\end{example}

\begin{remark} An immediate consequence of Example \ref{ex216} is that the (joint) semi-hyponormality of $(T_1,T_2)$ does not in general imply the semi-hyponormality of $T_1$ and $T_2$, since for unilateral weighted shifts semi-hyponormality is equivalent to hyponormality.
\end{remark}


\subsection{Commutativity of the entries of $L$}

As we have seen in Subsection \ref{commL}, the concrete calculation of the square root of $L$ depends on having the entries of $L$ commute. \ We will now obtain necessary and sufficient conditions for this to happen. 

\begin{lemma} \label{lem100} \ Let $W_{(\alpha,\beta)} \equiv (T_1,T_2)$ be a commuting $2$--variable weighted shift, and recall that 
$L = (\{T_j^*T_i\}_{i,j=1}^2$. \ Then
\newline (i) $T_1^*T_1$ commutes with $T_2^*T_1$ if and only if $\alpha_{\mathbf{k}+\mathbf{\varepsilon}_1} = \alpha_{\mathbf{k}+\mathbf{\varepsilon}_2}$, for all $\mathbf{k} \in \mathbb{Z}_+^2$;
\newline (ii) $T_2^*T_2$ commutes with $T_1^*T_2$ if and only if $\beta_{\mathbf{k}+\mathbf{\varepsilon}_1} = \beta_{\mathbf{k}+\mathbf{\varepsilon}_2}$, for all $\mathbf{k} \in \mathbb{Z}_+^2$;
\newline (iii) $T_1^*T_1$ always commutes with $T_2^*T_2$.
\end{lemma} 

\begin{proof}
(i) \ Since $T_2^* e_{\mathbf{k}}$ whenever $k_2=0$, it is enough to restrict attention to the action of the commutator $[T_1^*T_1,T_2^*T_1]$ on vectors $e_\mathbf{k}$ with $k_2 \ge 1$. \ Assume $k_2 \ge 1$; then 
\begin{eqnarray*}
[T_2^*T_1,T_1^*T_1]e_{\mathbf{k}} &=& \left(\beta_{\mathbf{k}+\mathbf{\varepsilon}_1-\mathbf{\varepsilon}_2}\alpha_{\mathbf{k}}^3 - \alpha_{\mathbf{k}+\mathbf{\varepsilon}_1-\mathbf{\varepsilon}_2}^2\beta_{\mathbf{k}+\mathbf{\varepsilon}_1-\mathbf{\varepsilon}_2}\alpha_{\mathbf{k}}\right)e_{\mathbf{k}+\mathbf{\varepsilon}_1-\mathbf{\varepsilon}_2} \\ 
&=& \alpha_{\mathbf{k}}\beta_{+\mathbf{\mathbf{k}+\varepsilon}_1-\mathbf{\varepsilon}_2}(\alpha_{\mathbf{k}}^2 - \alpha_{\mathbf{k}+\mathbf{\varepsilon}_1-\mathbf{\varepsilon}_2}^2)e_{\mathbf{k}+\mathbf{\varepsilon}_1-\mathbf{\varepsilon}_2}.
\end{eqnarray*} 
Therefore, $[T_2^*T_1,T_1^*T_1] = 0$ if and only if $\alpha_{\mathbf{k}+\varepsilon_1} = \alpha_{\mathbf{k}+\varepsilon_2}$, for all $\mathbf{k} \in \mathbb{Z}_+^2$. 

(ii) \ Since $T_1^* e_{\mathbf{k}}$ whenever $k_1=0$, it is enough to restrict attention to the action of the commutator $[T_2^*T_2,T_1^*T_2]$ on vectors $e_\mathbf{k}$ with $k_1 \ge 1$. \ Assume $k_1 \ge 1$; then 
\begin{eqnarray*}
[T_1^*T_2,T_2^*T_2]e_{\mathbf{k}} &=& \left(\alpha_{\mathbf{k}-\mathbf{\varepsilon}_1+\mathbf{\varepsilon}_2}\beta_{\mathbf{k}}^3 - \beta_{\mathbf{k}-\mathbf{\varepsilon}_1+\mathbf{\varepsilon}_2}^2\alpha_{\mathbf{k}-\mathbf{\varepsilon}_1+\mathbf{\varepsilon}_2}\beta_{\mathbf{k}}\right)e_{\mathbf{k}-\mathbf{\varepsilon}_1+\mathbf{\varepsilon}_2} \\ 
&=& \alpha_{\mathbf{k}-\mathbf{\varepsilon}_1+\mathbf{\varepsilon}_2}\beta_{\mathbf{k}}(\beta_{\mathbf{k}}^2 - \beta_{\mathbf{k}-\mathbf{\varepsilon}_1+\mathbf{\varepsilon}_2}^2)e_{\mathbf{k}-\mathbf{\varepsilon}_1+\mathbf{\varepsilon}_2}.
\end{eqnarray*} 
Therefore, $[T_1^*T_2,T_2^*T_2] = 0$ if and only if $\beta_{\mathbf{k}+\varepsilon_1} = \beta_{\mathbf{k}+\varepsilon_2}$, for all $\mathbf{k} \in \mathbb{Z}_+^2$. 

(iii) \ This is a straightforward calculation.
\end{proof}

\begin{corollary} \ Let $W_{(\alpha,\beta)} \equiv (T_1,T_2)$ be a commuting $2$--variable weighted shift, and recall that $L = (\{T_j^*T_i\}_{i,j=1}^2$, and assume that the entries of $L$ commute. \ Then 
\begin{equation} \label{identity100}
\alpha_{\mathbf{k}} \beta_{\mathbf{k + \varepsilon_2}} = \beta_{\mathbf{k}} \alpha_{\mathbf{k + \varepsilon_1}} \quad \quad (\textrm{all } \mathbf{k} \in \mathbb{Z}_+^2).
\end{equation}
\end{corollary}

\begin{proof} Let $\mathbf{k} \in \mathbb{Z}_+^2$ be arbitrary. \ Then
\begin{eqnarray*}
\alpha_{\mathbf{k}} \beta_{\mathbf{k + \varepsilon_2}} &=& \alpha_{\mathbf{k}} \beta_{\mathbf{k + \varepsilon_1}} \quad (\textrm{by  Lemma \ref{lem100}(ii)}) \\
&=& \alpha_{\mathbf{k + \varepsilon_2}}\beta_{\mathbf{k}} \quad (\textrm{by the commutativity of } W_{(\alpha,\beta)}) \\
&=& \alpha_{\mathbf{k + \varepsilon_1}} \beta_{\mathbf{k}} \quad (\textrm{by Lemma \ref{lem100}(i)}),
\end{eqnarray*}
as desired.
\end{proof}


\subsection{Commutativity of the entries of $R$}

In a way entirely similar to that used in the previous subsection, we now find necessary and sufficient conditions for the commutativity of the entries of $R$.

\begin{lemma} \label{lem101} \ Let $W_{(\alpha,\beta)} \equiv (T_1,T_2)$ be a commuting $2$--variable weighted shift, and recall that 
$R = (\{T_iT_j^*\}_{i,j=1}^2$. \ Then
\newline (i) $T_1T_1^*$ commutes with $T_1T_2^*$ if and only if $\alpha_{\mathbf{k}+\mathbf{\varepsilon}_2} = \alpha_{\mathbf{k}+\mathbf{\varepsilon}_1}$, for all $\mathbf{k} \in \mathbb{Z}_+^2$;
\newline (ii) $T_2T_1^*$ commutes with $T_2T_2^*$ if and only if $\beta_{\mathbf{k}+\mathbf{\varepsilon}_1} = \beta_{\mathbf{k}+\mathbf{\varepsilon}_2}$, for all $\mathbf{k} \in \mathbb{Z}_+^2$;
\newline (iii) $T_1T_1^*$ always commutes with $T_2T_2^*$.
\end{lemma} 
\begin{proof}
Using the same technique as in Lemma \ref{lem100}, we easily obtain
\begin{eqnarray*}
[T_1T_1^*,T_1T_2^*]e_{\mathbf{k}} &=& \left(\beta_{\mathbf{k - \varepsilon_2}}\alpha_{\mathbf{k - \varepsilon_2}}^3 - \alpha_{\mathbf{k - \varepsilon_1}}^2\beta_{\mathbf{k - \varepsilon_2}}\alpha_{\mathbf{k - \varepsilon_2}}\right)e_{\mathbf{k +\varepsilon_1 - \varepsilon_2}} \\ 
&=& \alpha_{\mathbf{k - \varepsilon_2}}\beta_{\mathbf{k - \varepsilon_2}}(\alpha_{\mathbf{k - \varepsilon_2}}^2 - \alpha_{\mathbf{k - \varepsilon_1}}^2)e_{\mathbf{k + \varepsilon_1 - \varepsilon_2}}.
\end{eqnarray*} 
Similarly,
\begin{eqnarray*}
[T_2T_2^*,T_2T_1^*]e_{\mathbf{k}} &=& \left(\alpha_{\mathbf{k - \varepsilon_1}}\beta_{\mathbf{k - \varepsilon_1}}^3 - \beta_{\mathbf{k - \varepsilon_2}}^2\alpha_{\mathbf{k - \varepsilon_1}}\beta_{\mathbf{k - \varepsilon_1}}\right)e_{\mathbf{k +\varepsilon_2 - \varepsilon_1}} \\ 
&=& \alpha_{\mathbf{k - \varepsilon_1}}\beta_{\mathbf{k - \varepsilon_1}}(\beta_{\mathbf{k - \varepsilon_1}}^2 - \beta_{\mathbf{k - \varepsilon_2}}^2)e_{\mathbf{k + \varepsilon_2 - \varepsilon_1}}.
\end{eqnarray*} 
It follows that $T_1T_1^*$ commutes with $T_1T_2^*$ if and only if $\alpha_{\mathbf{k}+\mathbf{\varepsilon}_2} = \alpha_{\mathbf{k}+\mathbf{\varepsilon}_2}$ and $T_2T_2^*$ commutes with $,T_2T_1^*$ if and only if $\beta_{\mathbf{k}+\mathbf{\varepsilon}_1} = \beta_{\mathbf{k}+\mathbf{\varepsilon}_2}$, for all $\mathbf{k} \in \mathbb{Z}_+^2$. \ Finally, as in Lemma \ref{lem100}(iii), the commutativity of $T_1T_1^*$ and $T_2T_2^*$ is immediate.
\end{proof}
We conclude this subsection with a result that combines Lemmas \ref{lem100} and \ref{lem101}.

\begin{corollary} \label{cor214} \ Let $W_{(\alpha,\beta)} \equiv (T_1,T_2)$ be a commuting $2$--variable weighted shift, and let $L$ and $R$ be as in Lemmas \ref{lem100} and \ref{lem101}. \ The following statements are equivalent:
\newline (i) \ the entries of $L$ commute;
\newline (ii) \ the entries of $R$ commute;
\newline (iii) \ $\alpha_{\mathbf{k}+\mathbf{\varepsilon}_2} = \alpha_{\mathbf{k}+\mathbf{\varepsilon}_1}$ and $\beta_{\mathbf{k}+\mathbf{\varepsilon}_1} = \beta_{\mathbf{k}+\mathbf{\varepsilon}_2}$, for all $\mathbf{k} \in \mathbb{Z}_+^2$.
\end{corollary}

We now characterize the $2$--variable weighted shifts for which $L$ and $R$ have commuting entries.

\begin{theorem} \ \ Let $W_{(\alpha,\beta)} \equiv (T_1,T_2)$ be a commuting $2$--variable weighted shift, let $L$ and $R$ be as in Lemmas \ref{lem100} and \ref{lem101}, and assume that both $L$ and $R$ have commuting entries. \ For $\ell,m \ge 0$, let $\omega_\ell := \alpha_{(\ell,0)}$ and $\eta_m := \beta_{(0,m)}$. \ Then $W_{(\alpha,\beta)}$ is the canonical embedding of $(W_\omega,W_\eta)$ (see Definition \ref{def26}).
\end{theorem} 

\begin{proof} Consider first $\alpha_{\mathbf{k}}$. \ By Corollary \ref{cor214}, 
$$
\alpha_{\mathbf{k}} = \alpha_{\mathbf{k}+\mathbf{\varepsilon}_1-\mathbf{\varepsilon}_2} = \ldots = \alpha_{\mathbf{k}+k_2\cdot \mathbf{\varepsilon}_1} = \alpha_{(k_1+k_2,0)} = \omega_{k_1+k_2}.
$$
Similarly,
$$
\beta_{\mathbf{k}} = \beta_{\mathbf{k}+\mathbf{\varepsilon}_2-\mathbf{\varepsilon}_1} = \ldots = \beta_{\mathbf{k}+k_1\cdot \mathbf{\varepsilon}_2} = \beta_{(0,k_1+k_2)} = \eta_{k_1+k_2}.
$$
It follows that $W_{(\alpha,\beta)}$ is indeed the canonical embedding of $(W_\omega,W_\eta)$.
\end{proof}

We will now establish that, for commuting canonical embeddings, the hyponormality (resp. subnormality) of $W_\omega$ and $W_\eta$ automatically implies the hyponormality (resp. subnormality) of the $2$--variable weighted shift $W_{(\alpha,\beta)}$. \ First, we need a lemma.

\begin{lemma} \label{lem223}
Let $W_{(\alpha,\beta)}$ be the canonical embedding of two subnormal unilateral weighted shifts $W_\omega$ and $W_\eta$, and assume that $W_{(\alpha,\beta)}$ is commuting. \ Then there exists a positive number $r$ such that 
$$
\eta_{k} = r \omega_k \quad (k \ge 0).
$$
\end{lemma}

\begin{proof}
From Remark \ref{rem27}(i), we know that $\omega_{\ell+1}\eta_{\ell} = \omega_{\ell}\eta_{\ell+1}$, for all $\ell \ge 0$. \ If we let 
$r:=\dfrac{\eta_0}{\omega_0}$, it follows that $\eta_{0}=r \omega_{0}$, $\eta_{1}=r \omega_{1}$, and by induction, $\eta_{\ell} = r \omega_{\ell} \; \, (\ell \ge 2)$. 
\end{proof}

\begin{theorem} \label{thm221}
Let $W_{(\alpha,\beta)}$ be the canonical embedding of two unilateral weighted shifts $W_\omega$ and $W_\eta$, and assume that $W_{(\alpha,\beta)}$ is commuting. \ Then \newline
(i)  $W_{(\alpha,\beta)} \textrm{ is hyponormal } \Longleftrightarrow W_{\omega} \textrm{ is hyponormal} \Longleftrightarrow W_{\eta} \textrm{ is hyponormal}$. \newline
(ii)  $W_{(\alpha,\beta)} \textrm{ is $k$--hyponormal } \!\Longleftrightarrow W_{\omega} \textrm{ is $k$--hyponormal}  \Longleftrightarrow W_{\eta} \textrm{ is $k$--hyponormal} \quad (1 \le k \in \mathbb{Z}_+$). \newline
(iii) \ $W_{(\alpha,\beta)} \textrm{ is subnormal } \Longleftrightarrow W_{\omega} \textrm{ is subnormal} \Longleftrightarrow W_{\eta} \textrm{ is subnormal}$.
\end{theorem}

\begin{proof}
Let $\gamma_{(k,\ell)}^{(\omega)}$, $\gamma_{(k,\ell)}^{(\eta)}$, and $\gamma_{(k,\ell)}^{(\alpha,\beta)}$ denote the moments of $W_\omega$, $W_{\eta}$, and $W_{(\alpha,\beta)}$, respectively. \ By commutativity, the moments $\gamma_{(k,\ell)}^{(\alpha,\beta)}$ of $W_{(\alpha,\beta)}$ are well-defined and can be computed as follows:
\begin{eqnarray} 
\gamma_{(0,0)}^{(\alpha,\beta)}&=&1, \nonumber \\
\gamma_{(k,0)}^{(\alpha,\beta)}&=&\alpha_{(0,0)}^2\alpha_{(1,0)}^2\cdot \ldots \cdot \alpha_{(k-1,0)}^2 \quad (k \ge 1), \nonumber \\
\gamma_{(0,\ell)}^{(\alpha,\beta)}&=&\beta_{(0,0}^2\beta_{(0,1)}^2\cdot \ldots \cdot \beta_{(0,k-1)}^2 \quad (\ell \ge 1), \textrm{ and} \nonumber \\
\gamma_{(k,\ell)}^{(\alpha,\beta)} &=& \alpha_{(0,0)}^2\alpha_{(1,0)}^2\cdot \ldots \cdot \alpha_{(k-1,0)}^2 \beta_{(k,0)}^2 \beta_{(k,1)}^2 \cdot \ldots \cdot \beta_{(k,\ell-1})^2 \quad (k, \ell \ge 1) \nonumber \\
&=&\omega_0^2 \omega_1^2 \cdot \ldots \cdot \omega_{k-1}^2 \eta_k^2 \eta_{k+1}^2 \cdot \ldots \cdot \eta_{k+\ell-1}^2 \nonumber \\
&=& \gamma_k^{(\omega)}\cdot \dfrac{\gamma_{k+\ell}^{(\eta)}}{\gamma_{k}^{(\eta)}} = 
\gamma_k^{(\omega)}\cdot \dfrac{r^{k+\ell}\gamma_{k+\ell}^{(\omega)}}{r^k\gamma_{k}^{(\omega)}} \quad (\textrm{using Lemma \ref{lem223}}) \nonumber \\
&=& r^\ell \gamma_{k+\ell}^{(\omega)}.     
\end{eqnarray}
(i) \ Using Lemma \ref{khypo}, Lemma \ref{khypon}, and (\ref{eq281}), consider the $3 \times 3$ matrix
\begin{equation} \label{1hypo}
M_{(p,q)}( W_{(\alpha ,\beta )}) \equiv \left(
\begin{array}{ccc}
\gamma_{(p,q)}^{(\alpha,\beta)} & \gamma _{(p,q)+\varepsilon_1}^{(\alpha,\beta)} & \gamma
_{(p,q)+\varepsilon_2}^{(\alpha,\beta)} \\
&& \\
\gamma _{(p,q)+\varepsilon_1}^{(\alpha,\beta)} & \gamma _{(p,q)+2\varepsilon_1}^{(\alpha,\beta)} & \gamma
_{(p,q)+\varepsilon_1+\varepsilon_2}^{(\alpha,\beta)} \\
&& \\
\gamma _{(p,q)+\varepsilon_2}^{(\alpha,\beta)} & \gamma _{(p,q)+\varepsilon_1+\varepsilon_2}^{(\alpha,\beta)} & \gamma
_{(p,q)+2\varepsilon_2}^{(\alpha,\beta)}
\end{array}
\right) 
= \left(
\begin{array}{ccc}
\gamma_{(p,q)}^{(\alpha,\beta)} & \gamma _{(p+1,\ell)}^{(\alpha,\beta)} & \gamma_{(p,q+1)}^{(\alpha,\beta)} \\
&& \\
\gamma _{(p+1,q)}^{(\alpha,\beta)} & \gamma _{(p+2,q)}^{(\alpha,\beta)} & \gamma
_{(p+1,q+1)}^{(\alpha,\beta)} \\
&& \\
\gamma _{(p,q+1)}^{(\alpha,\beta)} & \gamma _{(p+1,q+1)}^{(\alpha,\beta)} & \gamma
_{(p,q+2)}^{(\alpha,\beta)}
\end{array}
\right) \nonumber
\end{equation}

\begin{equation} \label{eq281}
= \left(\begin{array}{ccc}
r^q \gamma_{p+q}^{(\omega)} & r^q \gamma_{p+q+1}^{(\omega)} & r^{q+1} \gamma_{p+q+1}^{(\omega)} \\
&& \\
r^q \gamma_{p+q+1}^{(\omega)} & r^q \gamma_{p+q+2}^{(\omega)} & r^{q+1} \gamma_{pk+q+2}^{(\omega)} \\
&& \\
r^{q+1} \gamma_{p+q+1}^{(\omega)} & r^{q+1} \gamma_{p+q+2}^{(\omega)} & r^{q+2} \gamma_{p+q+2}^{(\omega)}
\end{array}
\right) 
\!\!=
\! r^q \cdot \left(
\begin{array}{ccc}
\gamma_{p+q}^{(\omega)} & \gamma_{p+q+1}^{(\omega)} & r \gamma_{p+q+1}^{(\omega)} \\
&& \\
\gamma_{p+q+1}^{(\omega)} & \gamma_{p+q+2}^{(\omega)} & r \gamma_{p+q+2}^{(\omega)} \\
&& \\
r \gamma_{p+q+1}^{(\omega)} & r \gamma_{p+q+2}^{(\omega)} & r^2 \gamma_{p+q+2}^{(\omega)}
\end{array}
\right).
\end{equation}

It is now evident that the last matrix is a flat extension of its compression to the first two rows and columns, namely a flat extension of the $2 \times 2$ matrix 
$\left(
\begin{array}{cc}
\gamma_{p+q}^{(\omega)} & \gamma_{p+q+1}^{(\omega)} \\
\gamma_{p+q+1}^{(\omega)} & \gamma_{p+q+2}^{(\omega)} 
\end{array}
\right)$, whose positive semidefiniteness is controlled by the hyponormality of $W_\omega$. \newline
(ii) \ To check $k$--hyponormality, we need to show the positive semidefiniteness of the matrix
$$
M_{(u,v)}(k):=\left(\gamma_{(u,v)+(m,n)+(p,q)}\right)_{\substack{0 \leqslant m+n \leqslant k \\ 0 \leqslant p+q \leqslant k}} \ge 0 
$$
for all $(u,v) \in \mathbb{Z}_{+}^2$. \ Now, if we label the columns of either matrix in (\ref{1hypo}) as $1,X,Y$, it is easy to see that $Y = r X$; this provides an alternative proof of the above-mentioned flatness. \newline
The recursive relation $Y = r X$ readily extends to the bigger moment matrices $M_{(u,v)}(k)$ introduced in Lemma \ref{khypon}. \ Concretely, if we label the rows and columns of $M_{(u,v)}(k)$ as $1,X,Y,X^2,XY, \linebreak Y^2,\ldots, XY^{k-1},Y^k$, we can show that the powers of $Y$ satisfy the identities $Y^\ell = r^\ell X^\ell$, for all $\ell \ge 1$. \ As a result, any mixed power $X^p Y^q$ can be written as $r^q \cdot X^{p+q}$, whenever $p,q \ge 1$. \ By successive and simultaneous transpositions of rows and columns, one can collect all powers of $X$ upfront, so that the upper-left $(k+1) \times (k+1)$ compression of the modified $M_{(u,v)}(k)$ is $M_u^{(\omega)}(k)$, i.e., the moment matrix of $W_\omega$ of order $k$ based at $p$. \ This is the prototypical moment matrix that detects the $k$--hyponormality of $W_\omega$. \newline  
Moreover, since all remaining columns include at least one power of $Y$, it follows that the entire matrix $M_{(u,v)}(k)$ is a flat extension of the moment matrix of $W_\omega$, as follows:
$$
M_{(u,v)}(k)=\left(
\begin{array}{ll}
M_u^{(\omega)}(k) & M_u^{(\omega)}(k)W \\
W^{\ast }M_u^{(\omega)}(k) & W^{\ast }M_u^{(\omega)}(k)W
\end{array}
\right).
$$
It follows that the $k$--hyponormality of $W_{(\alpha,\beta)}$ is equivalent to the $k$--hyponormality if $W_\omega$, which in turn is equivalent to the $k$-hyponormality of $W_\eta$, since one can write the Column $Y$ as $\dfrac{1}{r} \cdot X$. \newline
(iii) \ It is well known that, for unilateral and $2$--variable weighted shifts, subnormality is equivalent to $k$--hyponormality for every $k$, as established by the Bram-Halmos Theorem. \ In view of this, the assertion in (iii) follows easily from (ii).
\end{proof}

\begin{remark}
If $W_{(\alpha ,\beta )}$ is a canonical embedding, it is easy to verify that $L$ is always positive, and that the semi-hyponormality of $W_{(\alpha ,\beta )}$ implies the semi-hyponormality of both $W_\omega$ and $W_\eta$. \ Since for unilateral weighted shifts, the notions of hyponormality and semi-hyponormality agree, it then follows that both $W_\omega$ and $W_\eta$ are hyponormal.
\end{remark}

 
\subsection{The Class $W_{(\alpha,\beta)}(a,x,y)$}

In this subsection, we describe in detail our fundamental class of examples. \ Given three parameters $a$, $x$ and $y$, in the open unit interval $(0,1)$, we define the sequences $\alpha$ and $\beta$ as follows:
$$
\alpha_{(0,0)}:=x, \; \alpha_{(0,j)}:=a \, (\textrm{for all }j \ge 1), \; \alpha_{(i,j)}:=1 \, (\textrm{for all }i \ge 1, j \ge 0);
$$
$$
\beta_{(0,0)}:=y, \; \beta_{(i,0)}:=\dfrac{ay}{x} \, (\textrm{for all }i \ge 1), \; \beta_{(i,j)}:=1 \, (\textrm{for all }i \ge 1, j \ge 0).
$$
\medskip
For the reader's convenience, we replicate below the weight diagram for $W_{(\alpha,\beta)}(a,x,y)$, already shown in Figure 1(ii).
\setlength{\unitlength}{1mm} \psset{unit=1mm}
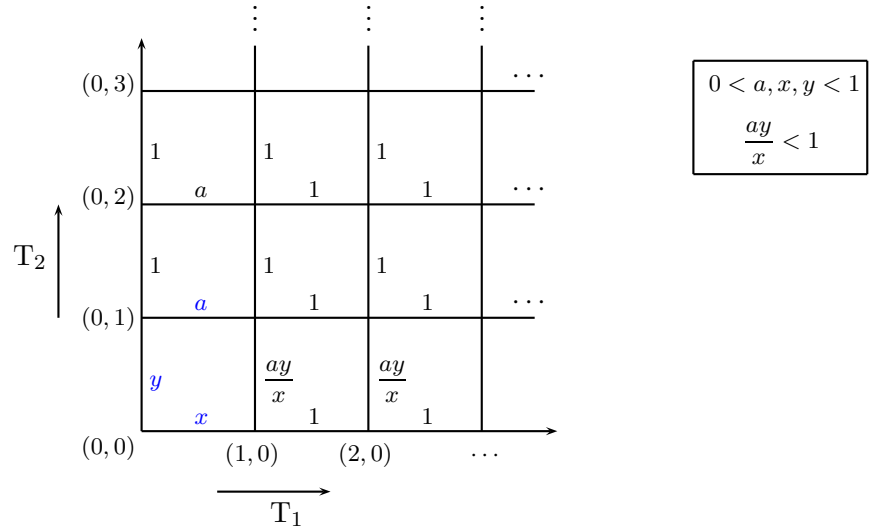
\begin{figure}[th]
\begin{center}
\begin{picture}(135,63)

\rput(-33,27){

\psline{->}(75,20)(130,20)
\psline(75,35)(127,35)
\psline(75,50)(127,50)
\psline(75,65)(127,65)

\psline{->}(75,20)(75,72)
\psline(90,20)(90,71)
\psline(105,20)(105,71)
\psline(120,20)(120,71)

\put(82,21){\footnotesize{$\cbl{x}$}}
\put(97,21){\footnotesize{$1$}}
\put(112,21){\footnotesize{$1$}}

\put(82,36){\footnotesize{$\cbl{a}$}}
\put(97,36){\footnotesize{$1$}}
\put(112,36){\footnotesize{$1$}}

\put(82,51){\footnotesize{$a$}}
\put(97,51){\footnotesize{$1$}}
\put(112,51){\footnotesize{$1$}}


\put(76,26){\footnotesize{$\cbl{y}$}}
\put(76,41){\footnotesize{$1$}}
\put(76,56){\footnotesize{$1$}}

\put(91,26){\footnotesize{$\dfrac{ay}{x}$}}
\put(91,41){\footnotesize{$1$}}
\put(91,56){\footnotesize{$1$}}

\put(106,26){\footnotesize{$\dfrac{ay}{x}$}}
\put(106,41){\footnotesize{$1$}}
\put(106,56){\footnotesize{$1$}}

\put(104.5,73){$\vdots$}
\put(89.5,73){$\vdots$}
\put(119.5,73){$\vdots$}

\put(124,36.1){$\cdots$}
\put(124,51.1){$\cdots$}
\put(124,66.1){$\cdots$}

\psline{->}(85,12)(100,12)
\put(92,8){$\rm{T}_1$}
\psline{->}(64,35)(64,50)
\put(58,42){$\rm{T}_2$}

\put(67,17){\footnotesize{$(0,0)$}}
\put(86,16){\footnotesize{$(1,0)$}}
\put(101,16){\footnotesize{$(2,0)$}}
\put(67,34){\footnotesize{$(0,1)$}}
\put(67,50){\footnotesize{$(0,2)$}}
\put(67,65){\footnotesize{$(0,3)$}}
\put(118.5,16){\footnotesize{$\cdots$}}


\put(150,65){\footnotesize{$0<a,x,y<1$}}
\put(154.5,57.5){\footnotesize{$\dfrac{ay}{x}<1$}}
\psline(148,69)(171,69)
\psline(148,54)(171,54)
\psline(148,54)(148,69)
\psline(171,54)(171,69)


}

\end{picture}
\end{center}
\caption{Weight diagram of the $2$--variable weighted shift $W_{(\alpha,\beta)}(a,x,y)$.}
\end{figure}

\begin{lemma}
Let $W_{(\alpha,\beta)}(a,x,y)$ be as above, and let $L$ be the associated $2 \times 2$ operator matrix. \ Then $L \ge 0$.
\end{lemma}

\begin{proof}
An inspection of the associated weight diagram reveals that the restrictions $W_{(\alpha,\beta)}(a,x,y)|_{\mathcal{M}}$ and $W_{(\alpha,\beta)}(a,x,y)|_{\mathcal{N}}$ are subnormal. \ Therefore, $L \ge 0$ if and only if the restriction of $L$ to the reducing subspace $\mathcal{K}(1)$ is positive, that is, if and only if
$$
\left(
\begin{array}{cccc}
a^2 & 0 & 0 & 0 \\
0 & 1 & \dfrac{a^2y}{x} & 0 \\
0 & \dfrac{a^2y}{x} & 1 & 0 \\
0 & 0 & 0 & \dfrac{a^2y^2}{x^2} 
\end{array}
\right).
$$
Since this matrix is positive, it follows that $L \ge 0$.  
\end{proof}


\subsection{Regions of Hyponormality and Subnormality} \label{Sub36}

From \cite[Proposition 2.10]{CuYo1}, we recall that $W_{(\alpha,\beta)}(a,x,y)$ is hyponormal if and only if 
\begin{equation} \label{hypotest}
y \le x \sqrt{\dfrac{1-x^2}{x^2-2a^2x^2+a^4}} = x \sqrt{\dfrac{1-x^2}{x^2(1-x^2)+(x^2-a^2)^2}}.
\end{equation}
On the other hand, $W_{(\alpha,\beta)}(a,x,y)$ is subnormal if and only if 
\begin{equation} \label{subntest}
y \le \sqrt{\dfrac{1-x^2}{1-a^2}} \quad (\textrm{by \cite[Proposition 2.11]{CuYo1}}).
\end{equation}
As a result, $W_{(\alpha,\beta)}(a,x,y)$ is hyponormal and {\it not subnormal} whenever
$$
\sqrt{\dfrac{1-x^2}{1-a^2}} < y \le x \sqrt{\dfrac{1-x^2}{x^2-2a^2x^2+a^4}} \quad (\textrm{by \cite[Theorem 2.12]{CuYo1}}),
$$
as visualized in Figure \ref{newfig} (cf. \cite[Figure 4]{CuYo1}). \ (To ensure that $T_2$ is hyponormal, we must require $\dfrac{ay}{x} \le 1$, i.e., $y \le \dfrac{x}{a}$.) 

\setlength{\unitlength}{0.8mm} \psset{unit=0.8mm}

\begin{figure}[th]
\begin{center}
\begin{picture}(120,76)

\rput(-20,35){
\pscustom[linewidth=0pt,fillstyle=solid,fillcolor=cyan]{
\pscurve(42,50.4)(47,48)(53,10)}
\pscustom[linewidth=0pt,fillstyle=solid,fillcolor=white]{
\pscurve(42,50.4)(48,38)(53,10)}

\psline{->}(10,10)(10,77)
\psline{->}(10,10)(62,75)
\psline[linewidth=2pt](10,10)(42,50)

\pscurve[linestyle=dashed,dash=3pt 2pt](10,71)(20,68)(42,50.4)(48,38)(53,10)
\pscurve[linestyle=dashed,dash=3pt 2pt](10,10)(28,40)(42,50.4)(47,48)(53,10)

\psline[linewidth=2pt]{<-}(46.9,44)(60,50)
\psline[linewidth=2pt]{<-}(40,34)(60,40)
\put(61,49){$\mathbf{T}$ is hyponormal but not subnormal in this region}
\put(61,39){$\mathbf{T}$ is subnormal in this region}
\psline[linewidth=2pt]{<-}(25,64.5)(27,74.5)
\put(24,79){$y={\sqrt{\frac{1-x^2}{1-a^2}}}$}
\psline[linewidth=2pt]{<-}(53.3,24)(63.3,27)
\put(65,27){$y=x{\sqrt{\frac{1-x^2}{x^2+a^4-2a^2x^2}}}$}
\put(6,6){\footnotesize{$(0,0)$}}
\put(40,5){\footnotesize{$(a,0)$}}
\put(42.5,10){\pscircle*(0,0){1}}
\put(1,41){\footnotesize{$(0,a)$}}
\put(10,42.5){\pscircle*(0,0){1}}
\put(52,10){\pscircle*(1,0){1}}
\put(51,5){\footnotesize{$(1,0)$}}
\put(63,75){$y={\frac{x}{a}}$}
\put(7,79){$y$}

\psline{->}(10,10)(75,10)
\psline[linewidth=2pt](10,10)(53,10)
\put(1,50){\footnotesize{$(0,1)$}}
\put(10,50.4){\pscircle*(0,0){1}}
\put(-11,70){\footnotesize{$(0,\sqrt{\frac{1}{1-a^2}})$}}

\put(77,7){$x$}
}

\end{picture}
\end{center}
\caption{Regions of hyponormality and subnormality for $W_{(\alpha,\beta)}(a,x,y)$.}
\label{newfig}
\end{figure}
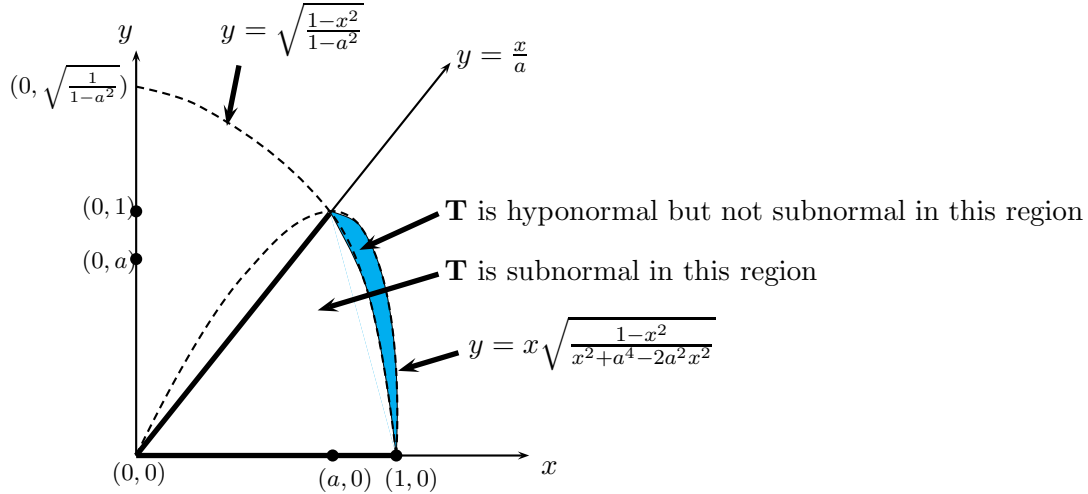

\begin{remark}
The necessary and sufficient conditions for the hyponormality (\ref{hypotest}) and subnormality (\ref{subntest}) can be rewritten with $a$ in terms of $x$ and $y$, as follows:

(i) \ $W_{(\alpha ,\beta )}(a,x,y)$ is jointly hyponormal if and only if
\begin{equation}
\begin{tabular}{l}
$\dfrac{x^{2}y-x\sqrt{\left( 1-x^{2}\right) \left( 1-y^{2}\right) }}{y}\leq
a^{2}\leq \dfrac{x^{2}y+x\sqrt{\left( 1-x^{2}\right) \left( 1-y^{2}\right) }}{%
y}$.%
\end{tabular}
\label{e1}
\end{equation}
(ii) \ $W_{(\alpha ,\beta )}(a,x,y)$ is jointly subnormal if and only if
\begin{equation}
\begin{tabular}{l}
$\dfrac{x^2+y^2-1}{y^2} \le a^2$.
\end{tabular}
\label{e2}
\end{equation}
In particular, observe that $W_{(\alpha,\beta)}(a,x,y)$ is always subnormal whenever $x^2+y^2<1$ and $ay < x$. 

These formulations will help us compare them to the necessary and sufficient condition for semi-hyponormality, which we now describe. 
\end{remark}

\subsection{Region of Semi-hyponormality} \label{Sub37}

\begin{theorem}
$W_{(\alpha,\beta)}(a,x,y)$ is jointly semi-hyponormal if and only if
\begin{equation}
\begin{tabular}{l}
$a^{2}\leq \dfrac{x\left(
4x^{2}-x^{4}+4x^{3}y+4y^{2}-6x^{2}y^{2}+4xy^{3}-y^{3}\right) }{4y\left(
x^{2}+y^{2}\right) }$ \\
\!\!\!\!\!\!\!\!\!\!\!\!and \\
 $\sqrt{\dfrac{\left( x^{2}+y^{2}\right) \left(
x+a^{2}y\right) }{x}}\geq \dfrac{\left( x+y\right) ^{2}}{2}+\sqrt{\dfrac{
\left( x+a^{2}y\right) \left( x-y\right) ^{4}}{4\left( x-a^{2}y\right) }}.$
\end{tabular}
\label{e3}
\end{equation}
\end{theorem}

\begin{proof}
We first observe that 
$$
L|_{\mathcal{K}(0)}=\left(
\begin{array}{ll}
x^{2} & 0 \\
0 & y^{2}
\end{array}
\right) \quad \textrm{ and } \quad R|_{\mathcal{K}(0)}=\left(
\begin{array}{ll}
0 & 0 \\
0 & 0
\end{array}
\right).
$$
Moreover, on the reducing subspace $\mathcal{K}(1)$, the relevant $2 \times 2$ summands that detect the positivity of $L$ and $R$ are
$$
\left(
\begin{array}{ll}
1 & \frac{a^{2}y}{x} \\
\frac{a^{2}y}{x} & 1
\end{array}
\right)
\quad \textrm{ and } \left(
\begin{array}{ll}
x^{2} & xy \\
xy & y^{2}
\end{array}
\right),
$$
respectively. \ Thus, we immediately get that $\sqrt{L|_{\mathcal{K}(0)}} \ge \sqrt{R|_{\mathcal{K}(0)}}$, and that 
$$
\sqrt{L|_{\mathcal{K}(1)}} \ge \sqrt{R|_{\mathcal{K}(1)}} 
\Longleftrightarrow 
\left(
\begin{array}{ll}
1 & \frac{a^{2}y}{x} \\
\frac{a^{2}y}{x} & 1%
\end{array}%
\right)^{\frac{1}{2}} \ge \left(
\begin{array}{ll}
x^{2} & xy \\
xy & y^{2}%
\end{array}
\right)^{\frac{1}{2}}.
$$
As a result, 
$$
W_{(\alpha,\beta)}(a,x,y) \textrm{ is jointly semi-hyponormal } \Longleftrightarrow \left(
\begin{array}{ll}
1 & \frac{a^{2}y}{x} \\
\frac{a^{2}y}{x} & 1%
\end{array}%
\right) ^{\frac{1}{2}}\geq \left(
\begin{array}{ll}
x^{2} & xy \\
xy & y^{2}%
\end{array}%
\right) ^{\frac{1}{2}}.
$$
To calculate the first square root, we resort to the orthogonal diagonalization. \ If we let \linebreak $U:=\left(
\begin{array}{cc}
-\frac{1}{\sqrt{2}} & \frac{1}{\sqrt{2}} \\
\frac{1}{\sqrt{2}} & \frac{1}{\sqrt{2}}
\end{array}
\right)$, it is straightforward to verify that 
\begin{equation*}
UL|_{\mathcal{K}(1)}U^*=\left(
\begin{array}{cc}
\frac{x-a^{2}y}{x} & 0 \\
0 & \frac{x+a^{2}y}{x}
\end{array}
\right) \quad \textrm{ and } \quad UR|_{\mathcal{K}(1)}U^*=\left(
\begin{array}{lc}
\frac{\left( x-y\right) ^{2}}{2} & \frac{y^{2}-x^{2}}{2} \\
\frac{y^{2}-x^{2}}{2} & \multicolumn{1}{l}{\frac{\left( x+y\right) ^{2}}{2}}
\end{array}
\right).
\end{equation*}
It follows that 
\begin{equation*}
U(L|_{\mathcal{K}(1)})^{\frac{1}{2}}U^* = \left(
\begin{array}{cc}
\sqrt{\frac{x-a^{2}y}{x}} & 0 \\
0 & \sqrt{\frac{x+a^{2}y}{x}}%
\end{array}%
\right) \quad \textrm{ and } \quad U(R|_{\mathcal{K}(1)})^{\frac{1}{2}}U^* = \frac{1}{\sqrt{x^{2}+y^{2}}}\left(
\begin{array}{cc}
\frac{\left( x-y\right) ^{2}}{2} & \frac{y^{2}-x^{2}}{2} \\
\frac{y^{2}-x^{2}}{2} & \multicolumn{1}{l}{\frac{\left( x+y\right) ^{2}}{2}}%
\end{array}
\right).
\end{equation*}
A simple calculation now shows that
\begin{equation*}
\begin{tabular}{l}
$(L|_{\mathcal{K}(1)})^{\frac{1}{2}}-(R|_{\mathcal{K}(1)})^{\frac{1}{2}}\geq 0$ \\
\\
$\Longleftrightarrow \operatorname{SH}(a,x,y):=\left(
\begin{array}{cc}
\sqrt{\dfrac{\left( x^{2}+y^{2}\right) \left( x-a^{2}y\right) }{x}}-\dfrac{%
\left( x-y\right) ^{2}}{2} & -\left( \dfrac{y^{2}-x^{2}}{2}\right)  \\
-\left( \dfrac{y^{2}-x^{2}}{2}\right)  & \sqrt{\dfrac{\left(
x^{2}+y^{2}\right) \left( x+a^{2}y\right) }{x}}-\frac{\left( x+y\right) ^{2}%
}{2}%
\end{array}%
\right) \geq 0$ \\
\\
$\Longleftrightarrow \sqrt{\dfrac{\left( x^{2}+y^{2}\right) \left(
x-a^{2}y\right) }{x}}-\dfrac{\left( x-y\right) ^{2}}{2}\geq 0 \quad$ and $\quad \det(SH(a,x,y))\geq 0$ \\
\\
$\Longleftrightarrow a^{2}\leq \dfrac{x\left(
4x^{2}-x^{4}+4x^{3}y+4y^{2}-6x^{2}y^{2}+4xy^{3}-y^{3}\right) }{4y\left(
x^{2}+y^{2}\right) }$ \quad and \\
\\
\hspace{20pt} $\sqrt{\dfrac{\left( x^{2}+y^{2}\right) \left(
x+a^{2}y\right) }{x}}\geq \dfrac{\left( x+y\right) ^{2}}{2}+\sqrt{\dfrac{%
\left( x+a^{2}y\right) \left( x-y\right) ^{4}}{4\left( x-a^{2}y\right) }}.$
\end{tabular}
\end{equation*}
\end{proof}


\subsection{Region of Weak Hyponormality} \label{Sub38}

Without loss of generality, we may assume that both $T_1$ and $T_2$ are hyponormal. \ As is well-known, the weak hyponormality of the commuting pair $(T_1,T_2)$ is controlled by the positivity of the self-commutators $C(\lambda):=[(T_1+\lambda T_2)^*,T_1+\lambda T_2]$, where $\lambda$ is a arbitrary complex number. \ Using the homogeneous orthogonal decomposition of $\ell^2(\mathbb{Z}_+^2)$, and the fact that each of the subspaces $\mathcal{K}(n)$ ($n \ge 0$) is reducing for $C(\lambda)$, it follows that $C(\lambda) \ge 0$ if and only if $C(\lambda)|_{\mathcal{K}(n)}$. \ Now, on $\mathcal{K}(0)$, 
$$
C(\lambda)e_{(0,0)} = (\alpha_{(0,0)} + \left|\lambda\right|^2 \beta_{(0,0)}) \, e_{(0,0)} \ge 0.
$$
while for $n > 1$, $C(\lambda)|_{\mathcal{K}(n)}$ is subnormal, and therefore weakly hyponormal. \ Thus, we need to consider in detail the restriction of $C(\lambda)$ to the reducing subspace $\mathcal{K}(1)$. \ 

For a fixed $\lambda \in \mathbb{C}$, consider the action of $C(\lambda)$ on a generic vector of $\mathcal{K}(1)$; that is,
\begin{equation*}
\left[ \left( T_{1}+\lambda T_{2}\right) ^{\ast },T_{1}+\lambda T_{2}\right]
\left( r \cdot e_{\left( 0,1\right) }+s \cdot e_{\left( 1,0\right) }\right) \quad \quad (r,s \in \mathbb{C}).
\end{equation*}
For $T_{1}+\lambda T_{2}$ to be hyponormal,
we need
\begin{equation*}
\left\langle \left[ \left( T_{1}+\lambda T_{2}\right) ^{\ast },T_{1}+\lambda
T_{2}\right] \left( r \cdot e_{\left( 0,1\right) }+s \cdot e_{\left( 1,0\right) }\right)
,\left( r \cdot e_{\left( 0,1\right) }+s \cdot e_{\left( 1,0\right) }\right) \right\rangle
\geq 0.
\end{equation*}
A simple calculation shows that this is equivalent to%
\begin{equation*}
\begin{tabular}{l}
$\left\langle \left[ T_{1}^{\ast },T_{1}\right] \left( r \cdot e_{\left( 0,1\right)
}+s \cdot e_{\left( 1,0\right) }\right) ,\left(r \cdot e_{\left( 0,1\right) }+s \cdot e_{\left(
1,0\right) }\right) \right\rangle $ \\
$+\lambda \left\langle \left[ T_{1}^{\ast },T_{2}\right] \left( r \cdot e_{\left(
0,1\right) }+s \cdot e_{\left( 1,0\right) }\right) ,\left( r \cdot e_{\left( 0,1\right)
}+s \cdot e_{\left( 1,0\right) }\right) \right\rangle $ \\
$+\overline{\lambda }\left\langle \left[ T_{2}^{\ast },T_{1}\right] \left(
r \cdot e_{\left( 0,1\right) }+s \cdot e_{\left( 1,0\right) }\right) ,\left( r \cdot e_{\left(
0,1\right) }+s \cdot e_{\left( 1,0\right) }\right) \right\rangle $ \\
$+\lambda \overline{\lambda }\left\langle \left[ T_{2}^{\ast },T_{2}\right]
\left( r \cdot e_{\left( 0,1\right) }+s \cdot e_{\left( 1,0\right) }\right) ,\left(
r \cdot e_{\left( 0,1\right) }+s \cdot e_{\left( 1,0\right) }\right) \right\rangle \geq 0.$%
\end{tabular}%
\end{equation*}%
Now recall that
$$
\begin{array}{rclrcl}
\left[ T_{1}^{\ast },T_{1}\right] e_{\left( 0,1\right) }&=&a^{2}e_{\left(
0,1\right) }\,; & \left[ T_{1}^{\ast },T_{1}\right] e_{\left( 1,0\right)
}&=&\left( 1-x^{2}\right) e_{\left( 1,0\right) }\,; \\
\left[ T_{2}^{\ast },T_{1}\right] e_{\left( 0,1\right) }&=&\left( \frac{a^{2}y%
}{x}-xy\right) e_{\left( 1,0\right) }\,; & \quad \left[ T_{2}^{\ast },T_{1}\right]
e_{\left( 1,0\right) }&=&0\, ; \\
\left[ T_{1}^{\ast },T_{2}\right] e_{\left(
0,1\right) }&=&0 \,; & \left[ T_{1}^{\ast },T_{2}\right] e_{\left( 1,0\right) }&=&\left( \frac{a^{2}y%
}{x}-xy\right) e_{\left( 0,1\right) }\,; \\
\left[ T_{2}^{\ast },T_{2}\right]
e_{\left( 0,1\right) }&=&\left( 1-y^{2}\right) e_{\left( 0,1\right) }\,; & \left[
T_{2}^{\ast },T_{2}\right] e_{\left( 1,0\right) }&=&\frac{a^{2}y^{2}}{x^{2}}e_{\left( 1,0\right) }\,.
\end{array}
$$
A calculation now shows that
\begin{equation*}
\begin{tabular}{l}
$\left\langle \left[ \left( T_{1}+\lambda T_{2}\right) ^{\ast
},T_{1}+\lambda T_{2}\right] \left( r \cdot e_{\left( 0,1\right) }+s \cdot e_{\left(
1,0\right) }\right) ,\left( r \cdot e_{\left( 0,1\right) }+s \cdot e_{\left( 1,0\right)
}\right) \right\rangle $ \\
$=(\left\vert r\right\vert ^{2}\left( 1-y^{2}\right) +\left\vert
s\right\vert ^{2}\frac{a^{2}y^{2}}{x^{2}})\left\vert \lambda \right\vert
^{2}+2\re(\overline{r}s\lambda (\frac{a^{2}y}{x}-xy))+\left\vert
r\right\vert ^{2}a^{2}+\left( 1-x^{2}\right) \left\vert s\right\vert ^{2}.$%
\end{tabular}
\end{equation*}
This is a complex quadratic form: 
\begin{equation} \label{quadraticform}
A(r,s)\left\vert \lambda \right\vert ^{2}+2\re(\lambda B\left( r,s\right))
+C\left( r,s\right) ,
\end{equation}
where 
$$
A\left( r,s\right) :=\left\vert r\right\vert ^{2}\left( 1-y^{2}\right)
+\left\vert s\right\vert ^{2}\frac{a^{2}y^{2}}{x^{2}}), 
$$
$$
B\left( r,s\right)
:=\overline{r}s(\frac{a^{2}y}{x}-xy),
$$
and
$$
C\left( r,s\right) :=\left\vert r\right\vert ^{2}a^{2}+\left( 1-x^{2}\right) \left\vert s\right\vert ^{2}.
$$
Since $A(r,s)$ and $C(r,s)$ are nonnegative for all values of $r$ and $s$, it readily follows that the nonnegativity of the quadratic form (\ref{quadraticform}) is completely determined by the non-positivity of its discriminant. \ To focus on nonnegativity, we consider minus the discriminant. \ We thus have, for all $r,s\in \mathbb{C}$ and a fixed $\lambda \in \mathbb{C}$,
$$
A\left( r,s\right) \left\vert \lambda \right\vert ^{2}+2\operatorname{Re}\left(
\lambda B\left( r,s\right) \right) +C\left( r,s\right) \geq 0 \Longleftrightarrow A\left(r,s\right) C\left( r,s\right) - \left\vert B\left( r,s\right) \right\vert ^{2} \ge 0.
$$
Therefore, $W_{(\alpha ,\beta)}(a,x,y)$ is weakly hyponormal if and only if 
\begin{equation}
\left( \left\vert r\right\vert ^{2}\left(
1-y^{2}\right) +\left\vert s\right\vert ^{2}\frac{a^{2}y^{2}}{x^{2}})\right) \left( \left\vert r\right\vert ^{2}a^{2}+\left( 1-x^{2}\right) \left\vert
s\right\vert ^{2}\right)
-\left\vert \overline{r}s\right\vert ^{2}\left( \frac{a^{2}y}{x}-xy\right)
^{2}\geq 0\text{.}  \label{imp}
\end{equation}
Observe that when $r$ or $s$ is equal to zero, $B(r,s)$ is zero, and hence the quadratic form is nonnegative. \ As a result, and without loss of generality, we may assume that $r \ne 0$ and $s \ne 0$. \ Let $p:=\left\vert r\right\vert$, $q:=\left\vert s\right\vert$, and $t:=\dfrac{p}{q}$. \ A straightforward calculation using {\it Mathematica} \cite{Wol} shows that (\ref{quadraticform}) is nonnegative if and only if
\begin{equation}
\left( \left( 1-y^{2}\right) a^{2}x^{2}\right) t^{4}+\left(
1-x^{2}-y^{2}+2a^{2}y^{2}\right) x^{2} t^{2}+a^{2}y^{2}\left(
1-x^{2}\right) \geq 0  \label{cal9}
\end{equation}
for all $t>0$. \ Since this is a bi-quadratic in $t$, and the free term
and the coefficient of $t^{4}$ are positive, we consider the sign of the
coefficient in $t^{2}$. \ Let
\begin{equation*}
f(a,x,y) :=\left( 1-x^{2}-y^{2}+2a^{2}y^{2}\right) x^{2}=\left(
1-x^{2}\right) x^{2}+\left( -1+2a^{2}\right) x^{2}y^{2}.
\end{equation*}%
If $f(a,x,y)$ is nonnegative, then the bi-quadratic is nonnegative (since $t^{2}$ is always nonnegative). \ If $f(a,x,y)$ is negative,
then we need to appeal to the discriminant of the bi-quadratic. \ Thus, we need to focus attention on the region 
$$
R_f := \{(a,x,y) \in (0,1)^3: f(a,x,y) < 0 \},
$$
and study the sign of the bi-quadratic’s discriminant in $R_f$. \ If $\sqrt{\frac{1}{2}}\leq a < 1$, then $f(a,x,y)>0$ for all $0<x,y<1$, so weak hyponormality is guaranteed. \ Thus, the relevant region under consideration is $R_f \bigcap \{(a,x,y) \in (0,1)^3: a<\sqrt{\frac{1}{2}} \}$. \ \newline
On the other hand, when $a < \sqrt{\frac{1}{2}}$, note that $f\left( a,x,y\right) \geq 0$ if and only if 
$y\leq \sqrt{\frac{1-x^{2}}{1-2a^{2}}}$. \ As a consequence, we will further restrict the domain of $f$ to those points $(a,x,y)$ with $0<a<\sqrt{\frac{1}{2}}$ and $y>\sqrt{\frac{1-x^{2}}{1-2a^{2}}}$. \ For later use, we record here the fact that in this region, a point $(x,y)$ is automatically outside the closed unit disk; for, 
$$
y > \sqrt{\frac{1-x^{2}}{1-2a^{2}}} \Longrightarrow y > \sqrt{1-x^2} \Longrightarrow x^2+y^2 > 1.
$$
We also recall that the weight diagram of $W_{(\alpha,\beta)}(a,x,y)$ requires $ay < x$, and therefore, by combining this condition and $y>\sqrt{\frac{1-x^{2}}{1-2a^{2}}}$, we obtain a necessary condition for x; that is, $x > ay > a \sqrt{\frac{1-x^{2}}{1-2a^{2}}}$, from which it follows that $x^2(1-2a^2)>a^2(1-x^2)$, so that $x^2(1-a^2)>a^2$, and therefore $x>\frac{a}{\sqrt{1-a^2}}$. \ We conclude that the region of interest is 
$$
\mathfrak{R} := \left\{\, (a,x,y): \; a < \sqrt{\frac{1}{2}}, \; \; x>\frac{a}{\sqrt{1-a^2}}, \; \textrm{ and } \; \frac{x}{a}> y>\sqrt{\frac{1-x^{2}}{1-2a^{2}}} \; \right\}.
$$
Outside of this region, weak hyponormality is guaranteed. \ 

We are now ready to study the bi-quadratic's discriminant. \ Since we aim for a non-positive discriminant, we will consider the additive inverse of the discriminant, which we will denote as $\operatorname{minusbiqdiscr}$. \ From (\ref{cal9}), we see that
$$
\operatorname{minusbiqdiscr} := 4a^4x^2y^2(1-y^{2})(1-x^{2})-x^4(1-x^{2}-y^{2}+2a^{2}y^{2})^2.
$$
A straightforward calculation using {\it Mathematica} \cite{Wol} shows that 
$$
\operatorname{minusbiqdiscr} = -x^2 (-1 + x^2 + y^2) (-x^2 + x^4 + 4 a^4 y^2 + x^2 y^2 - 4 a^2 x^2 y^2).
$$
Now recall that the region of interest excludes the closed unit disk. \ Therefore, 
$$
\operatorname{minusbiqdiscr} \ge 0 \Longleftrightarrow g(a,x,y) := -x^2 + x^4 + 4 a^4 y^2 + x^2 y^2 - 4 a^2 x^2 y^2 \le 0.
$$
We now observe that $g(a,x,y)=Py^2-Q$, where 
$$
P:=x^2(1-x^2)+(2a^2-x^2)^2 \quad \textrm{ and } \quad Q:=x^2(1-x^2).
$$
It follows that 
$$
g(a,x,y) \le 0 \Longleftrightarrow y \le \operatorname{weakhyp}:=\sqrt{\dfrac{Q}{P}} = \dfrac{x \sqrt{1-x^2}}{\left|2a^2-x^2\right|}.
$$

\medskip
The preceding discussion serves as a proof of our main result of this subsection, which we now state.

\begin{theorem} \ Let $(a,x,y) \in \mathfrak{R}$. \ Then 
$$
W_{(\alpha,\beta)}(a,x,y) \textrm{ is weakly hyponormal } \Longleftrightarrow y \le \dfrac{x \sqrt{1-x^2}}{\left|2a^2-x^2\right|}.
$$
\end{theorem}
\bigskip

\subsection{Partition of the Open Unit Cube $(0,1)^3$ into Sub-regions}

The results in Subsections \ref{Sub36}, \ref{Sub37} and \ref{Sub38} will allow us to draw a partition of the open unit cube $(0,1)^3$ into regions and sub-regions which describe the structural properties of the class $\mathcal{W}$; that is, we will describe in detail the conditions under which a $2$--variable weighted shift $W_{(\alpha,\beta)}(a,x,y)$ is subnormal, hyponormal, semi-hyponormal, and weakly hyponormal. \ In Figure 7, we capture all the different possibilities in a graph. \ 

First, since the significant regions of interest appear when $a<\sqrt{\frac{1}{2}}$, and the relative positions of those regions and their sub-regions do not change as we move $a$, we fix $a$ to be $\dfrac{1}{2}$ and we focus on the pairs $(x,y)$ lying in the open rectangle $(0.45, 0.66) \times (0.95, 1.00)$. \ For instance, the region below the blue curve in Figure 7 corresponds to the subnormality of $W_{(\alpha,\beta)}(a,x,y)$, while for $(x,y)$ in the region between the blue and the red curves, the $2$--variable weighted shift is hyponormal but not subnormal. \ Similarly, the region above the red curve describes the non-hyponormality of $W_{(\alpha,\beta)}(a,x,y)$. \ Within that region, we identify two curves (in green and brown) which break the region of non-hyponormality into four sub-regions: $SH \sim H$ denotes the set of pairs $(x,y)$ that give rise to a semi-hyponormal $W_{(\alpha,\beta)}(a,x,y)$ which is not hyponormal; similarly, $WH \sim SH$ stands for weakly hyponormal and not hyponormal, $\sim SH \; \& \;\sim WH$ represents the case non-hyponormal and non-semi-hyponormal, and finally, $(SH \; \& \; WH) \sim H$ stands for non-hyponormal $W_{(\alpha,\beta)}(a,x,y)$ which are simultaneously semi-hyponormal and weakly hyponormal.

As a result, for an arbitrary point $(a,x,y) \in (0,1)^3$, we can precisely identify whether the associated $W_{(\alpha,\beta)}(a,x,y)$ is subnormal, hyponormal, semi-hyponormal or weakly hyponormal. \ 

Finally, the reader may have already noticed that, inside the quadrant $x^2+y^2<1$, and irrespective of the value of $a$ that satisfies $ay<x$, the $2$--variable weighted shift $W_{(\alpha,\beta)}(a,x,y)$ is always subnormal. 

\begin{figure}[th]
\includegraphics{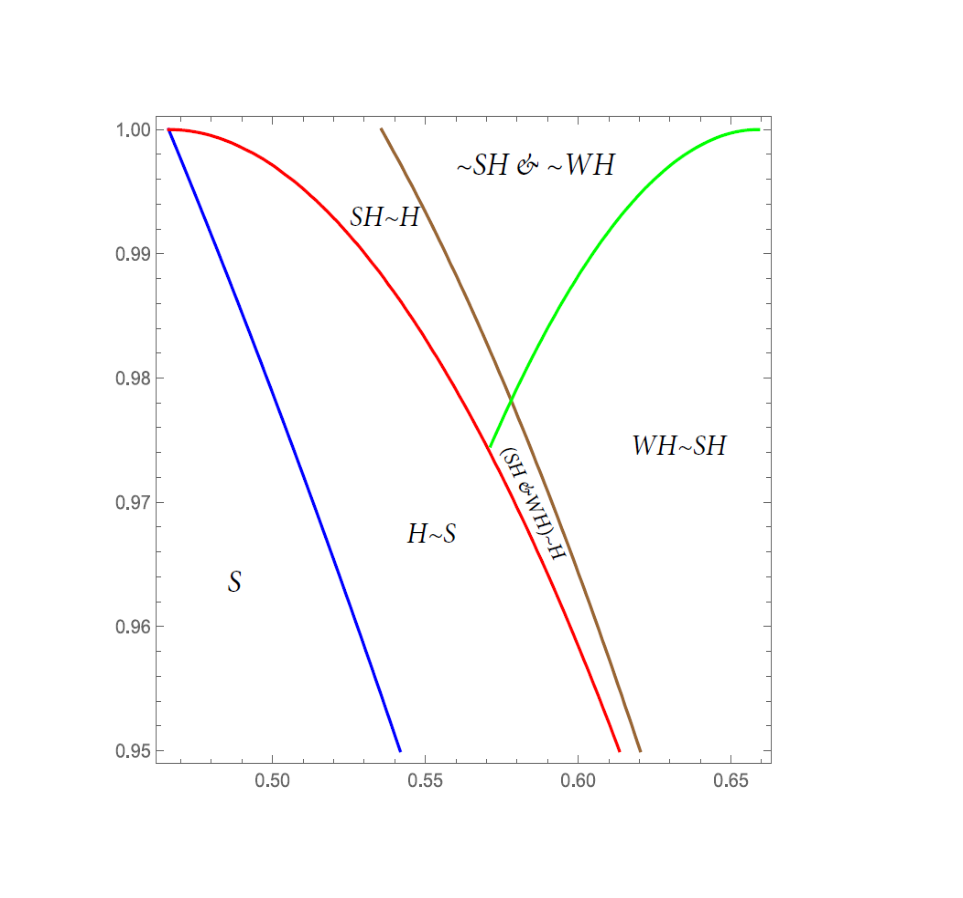}
\caption{Partition of the open unit cube $(0,1)^3$ into sub-regions. \ The parameters $x$ and $y$ are represented on the horizontal and vertical axes, respectively.}
\end{figure}


\subsection{The Case of the Drury-Arveson Shift}

We first recall the weight diagram for the Drury-Arveson $2$--variable weighted shift $DA \equiv (T_1,T_2)$, shown in Figure 8. \ As is well known, $DA$ is a commuting pair of mutually unitarily equivalent subnormal operators $T_1$ and $T_2$. \ Moreover, for a fixed $\ell \ge 1$, the restriction of $T_1$ to the reducing subspace $\bigvee \{e_{(k.\ell)}: k \ge 0 \}$ is the Agler shift $A_{\ell}\equiv \shift(\sqrt{\frac{1}{\ell}}, \sqrt{\frac{2}{\ell+1}}, \sqrt{\frac{3}{\ell+2}}, \ldots)$. \ It is also well-known that $DA$ is not (jointly) hyponormal, as proved in \cite[Section 9]{CuYo6}. \ We will now establish that $DA$ is not even semi-hyponormal.

\setlength{\unitlength}{1mm} \psset{unit=1mm}

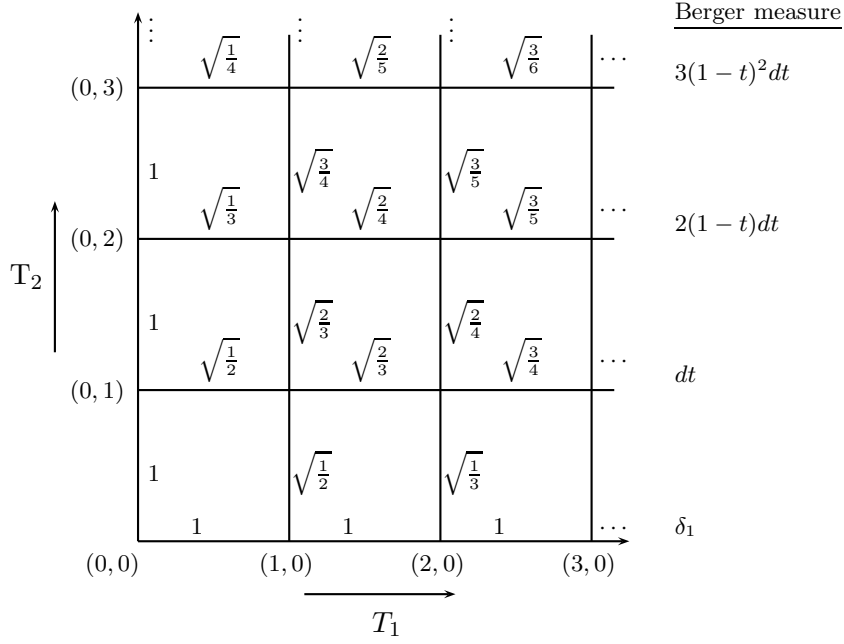
\begin{figure}[th]
\begin{center}
\begin{picture}(150,90)

\rput(-10,40){

\psline{->}(67,13)(87,13)
\put(76,8){${T}_1$}
\psline{->}(34,45)(34,65) 
\put(28,54){$\rm{T}_2$}
\put(36,39){\footnotesize{$(0,1)$}}
\put(36,59){\footnotesize{$(0,2)$}}
\put(36,79){\footnotesize{$(0,3)$}}

\psline{->}(45,20)(110,20)
\psline(45,40)(108,40)
\psline(45,60)(108,60)
\psline(45,80)(108,80)

\psline{->}(45,20)(45,90)
\psline(65,20)(65,87)
\psline(85,20)(85,87)
\psline(105,20)(105,87)

\put(38,16){\footnotesize{$(0,0)$}}
\put(61,16){\footnotesize{$(1,0)$}}
\put(81,16){\footnotesize{$(2,0)$}}
\put(101,16){\footnotesize{$(3,0)$}}

\put(52,21){\footnotesize{$1$}}
\put(72,21){\footnotesize{$1$}}
\put(92,21){\footnotesize{$1$}}
\put(106,21){\footnotesize{$\cdots$}}
\put(116,21){\footnotesize{$\delta_1$}}

\put(53,43){\footnotesize{$\sqrt{\frac{1}{2}}$}}
\put(73,43){\footnotesize{$\sqrt{\frac{2}{3}}$}}
\put(93,43){\footnotesize{$\sqrt{\frac{3}{4}}$}}
\put(106,43){\footnotesize{$\cdots$}}
\put(116,41){\footnotesize{$dt$}}

\put(53,63){\footnotesize{$\sqrt{\frac{1}{3}}$}}
\put(73,63){\footnotesize{$\sqrt{\frac{2}{4}}$}}
\put(93,63){\footnotesize{$\sqrt{\frac{3}{5}}$}}
\put(106,63){\footnotesize{$\cdots$}}
\put(116,61){\footnotesize{$2(1-t)dt$}}

\put(53,83){\footnotesize{$\sqrt{\frac{1}{4}}$}}
\put(73,83){\footnotesize{$\sqrt{\frac{2}{5}}$}}
\put(93,83){\footnotesize{$\sqrt{\frac{3}{6}}$}}
\put(106,83){\footnotesize{$\cdots$}}
\put(116,81){\footnotesize{$3(1-t)^2dt$}}

\put(116,89){\footnotesize{\underline{Berger measure}}}


\put(46.3,28){\footnotesize{$1$}}
\put(46.3,48){\footnotesize{$1$}}
\put(46.3,68){\footnotesize{$1$}}
\put(46.3,86){\footnotesize{$\vdots$}}

\put(65.2,28){\footnotesize{$\sqrt{\frac{1}{2}}$}}
\put(65.3,48){\footnotesize{$\sqrt{\frac{2}{3}}$}}
\put(65.3,68){\footnotesize{$\sqrt{\frac{3}{4}}$}}
\put(66,86){\footnotesize{$\vdots$}}

\put(85.2,28){\footnotesize{$\sqrt{\frac{1}{3}}$}}
\put(85.3,48){\footnotesize{$\sqrt{\frac{2}{4}}$}}
\put(85.3,68){\footnotesize{$\sqrt{\frac{3}{5}}$}}
\put(86,86){\footnotesize{$\vdots$}}

}

\end{picture}
\end{center}
\caption{Weight diagram of the Drury-Arveson shift.}
\end{figure}

\begin{theorem}
The Drury-Arveson shift is not semi-hyponormal.
\end{theorem}

\begin{proof}
To show that $DA$ is not semi-hyponormal, it is enough to establish that the restriction of $\sqrt{L} - \sqrt{R}$ to the reducing subspace $\mathcal{K}(1)$ is not a positive semi-definite $2 \times 2$ matrix. \ Since
$$
L|_{\mathcal{K}(1)} = \left(
\begin{array}{cc}
\alpha_{(1,0)}^2 & \alpha_{(0,1)}\beta_{(1,0)} \\
\alpha_{(0,1)}\beta_{(1,0)} & \beta_{(1,0)}^2
\end{array}
\right) = \left(
\begin{array}{cc}
1 & \frac{1}{2} \\
\frac{1}{2} & 1
\end{array}
\right),  
$$
we can use (\ref{squareroot2}) to obtain
$$
\sqrt{L|_{\mathcal{K}(1)}} = \dfrac{1}{\sqrt{2 + \sqrt{3}}} \cdot \left(
\begin{array}{cc}
1 + \frac{\sqrt{3}}{2} & \frac{1}{2} \\
\frac{1}{2} & 1 + \frac{\sqrt{3}}{2}
\end{array}
\right). 
$$
Similarly,
$$
R|_{\mathcal{K}(1)} = \left(
\begin{array}{cc}
\alpha_{(0,0)}^2 & \alpha_{(0,0)}\beta_{(0,0)} \\
\alpha_{(0,0)}\beta_{(0,0)} & \beta_{(0,0)}^2
\end{array}
\right) = \left(
\begin{array}{cc}
1 & 1 \\
1 & 1
\end{array}
\right)  
$$
and 
$$
\sqrt{R|_{\mathcal{K}(1)}} = \left(
\begin{array}{cc}
\sqrt{\frac{1}{2}} & \sqrt{\frac{1}{2}} \\
\sqrt{\frac{1}{2}} & \sqrt{\frac{1}{2}}
\end{array}
\right).  
$$
Therefore,
$$
\sqrt{L|_{\mathcal{K}(1)}} - \sqrt{R|_{\mathcal{K}(1)}} = \dfrac{1}{\sqrt{2 + \sqrt{3}}} \cdot \left(
\begin{array}{cc}
1 + \frac{\sqrt{3}}{2} & \frac{1}{2} \\
\frac{1}{2} & 1 + \frac{\sqrt{3}}{2}
\end{array}
\right) - \left(
\begin{array}{cc}
\sqrt{\frac{1}{2}} & \sqrt{\frac{1}{2}} \\
\sqrt{\frac{1}{2}} & \sqrt{\frac{1}{2}}
\end{array}
\right).  
$$
It is now straightforward to verify that
$$
\det(\sqrt{L|_{\mathcal{K}(1)}} - \sqrt{R|_{\mathcal{K}(1)}}) \cong -0.133975 < 0.  \qed
$$
\end{proof}


\end{document}